\documentclass[12pt,a4paper,twoside,final,notitlepage, reqno]{article}

\usepackage[english]{babel}
\usepackage[latin1]{inputenc}
\usepackage[a4paper,left=3.5cm,right=3.5cm]{geometry}
\usepackage{amssymb}
\usepackage{amsthm} 
\usepackage{amsmath}
\usepackage{mathtools}
\usepackage{amscd}
\usepackage{geometry}
\usepackage{graphics,graphicx}
\usepackage{epstopdf}
\usepackage[usenames, dvipsnames]{color}
\usepackage[textsize=small]{todonotes}
\usepackage{booktabs}
\usepackage{subfigure} 
\usepackage{siunitx}

\graphicspath{{figs/}}
\makeatletter
\def\input@path{{figs/}}
\makeatother

\theoremstyle{definition}
\newtheorem{theorem}{Theorem}%[section]
\newtheorem{lemma}{Lemma}
\newtheorem{proposition}{Proposition}
\newtheorem{corollary}{Corollary}

\theoremstyle{definition}
\theoremstyle{definition}
\newtheorem{remark}{Remark}%[section]

\providecommand{\keywords}[1]  {\textbf{Keywords:} #1}

\providecommand{\acknow}[1] {\textbf{Acknowledgments.} #1}

\let\div\undefined
\DeclareMathOperator{\div}{div} 

\DeclareMathOperator{\e}{e}

\DeclareMathOperator{\diag}{diag}

\usepackage{esdiff}
\renewcommand{\d}[1]{\, \mathrm{d} #1}

\newcommand{\R}{\mathbb{R}}
\newcommand{\C}{\mathbb{C}}

\renewcommand{\phi}{\varphi}

\def\sizeg{ \footnotesize}

\newcommand{\Iion} {I_\text{ion}}
\newcommand{\Ist} {I_\text{st}}
\newcommand{\vref} {v_\text{ref}}

\def \Pk {\mathbb{P}_{k-1}} 

\def \xt{\tilde{x}_{_{[t,Y,h]}}} 
\def \gty{g_{_{t, y_k}}}
\def \gt{\bar{g}_{_{[t,Y,h]}}} 
\def \at{\tilde{a}_{_{[t,Y,h]}}} 
\def \bt{\tilde{b}_{_{[t,Y,h]}}}

\def \L{\mathcal{L}}

\usepackage{authblk}

\title{
  \vspace{-3.4cm}
  \bf{\Large{
      \mbox{
        \!\!\!\!\!\!\!\!\!\!\!\!\!\!\!\!
        Exponential Adams-Bashforth integrators for stiff ODEs,
      }       
      \\[5pt]
      application to cardiac electrophysiology
  }}
}

\author[1]{Yves Coudi\`ere\thanks{yves.coudiere@inria.fr}}
\author[1]{Charlie Douanla-Lontsi\thanks{charlie.douanla-lontsi@inria.fr}}
\author[2]{Charles Pierre\thanks{charles.pierre@univ-pau.fr}}

\affil[1]{
  INRIA Bordeaux, 
  Institut de Math\'ematiques de Bordeaux, 
  \protect \\
  UMR CNRS 5241, Universit\'e de Bordeaux, France.
}\affil[2]{
  Laboratoire  de Math\'ematiques et de leurs Applications,
  UMR CNRS 5142, \protect \\
  Universit\'e de Pau et des Pays de l'Adour, France.
  }

\usepackage{fancyhdr}

\begin{document} 

\date{April, 2018}

\maketitle

\vspace{20pt}
\noindent
\keywords{ 
  stiff equations,  
  explicit high-order multistep methods, 
  exponential integrators, 
  stability and convergence, Dahlquist stability
}
\\[8pt]
\acknow{
This study received financial support from the French Government as part of the
``Investissement d'avenir'' program managed by the National Research Agency
(ANR), Grant reference ANR-10-IAHU-04. It also received fundings of the project
ANR-13-MONU-0004-04.
}
% \subjclass{
%   65L04, 65L06, 65L20, 65L99  
% }
\\[8pt]
  \begin{abstract}
    %% Text of abstract
    Models in cardiac electrophysiology are coupled systems of reaction diffusion
    PDE and of ODE. The ODE system displays a very stiff behavior. It is non
    linear and its upgrade at each time step is a preponderant load in the
    computational cost. The issue is to develop high order explicit and stable
    methods to cope with this situation.

    In this article is analyzed the resort to exponential Adams Bashforth
    (EAB) integrators in cardiac electrophysiology. The method is presented in the
    framework of a general and varying stabilizer, that is well suited in this
    context. Stability under perturbation (or 0-stability) is proven. It provides
    a new approach for the convergence analysis of the method. The Dahlquist
    stability properties of the method is performed. It is presented in a new
    framework that incorporates the discrepancy between the stabilizer and the
    system Jacobian matrix. Provided this discrepancy is small enough, the method
    is shown to be A(alpha)-stable. This result is interesting for an explicit
    time-stepping method. Numerical experiments are presented for two classes of
    stiff models in cardiac electrophysiology. They include performances
    comparisons with several classical methods. The EAB method is observed to be
    as stable as implicit solvers and cheaper at equal level of accuracy.
    \\$~$
  \end{abstract}

\section*{Introduction}
\label{sec:intro}
Computations in cardiac electrophysiology have to face two constraints.
Firstly the stiffness due to heterogeneous time and space scales.
This is usually dealt with by considering very fine grids. 
That strategy is associated with large computational costs, still challenging in dimension 3.
Secondly, the resolution of the reaction terms that appears in the ionic models
has an important cost. 
That resolution occur at each grid node.
The total amount of evaluation of the reaction terms has to be maintained as low as possible by the considered numerical method.
Implicit solvers therefore are usually avoided.
Exponential integrators are well adapted to cope with these two constraints.
Actually they both allow an explicit resolution of the reaction term and display strong stability properties.
In this article is studied and analyzed exponential time-stepping methods
dedicated to the resolution of  reaction equations.

Models for the propagation of the cardiac action
potential are evolution reaction diffusion equations
coupled with ODE systems. The widely used monodomain model
\cite{CLEMENS-NENONEN-04, FRANZONE-PAVARINO-TACCARDI-05.1,
  FRANZONE-PAVARINO-TACCARDI-05.2} formulates as,
\begin{displaymath}
  \frac{\partial v}{\partial t} = Av + f_1(v,w,x,t),
  \quad \frac{\partial w}{\partial t} = f_2(v,w,x,t),
\end{displaymath}
with space and time variables $x\in\Omega\subset \R^d$ and $t\in\R$. The
unknowns are the functions $v(t,x)\in\R$ (the transmembrane voltage) and
$w(t,x)\in\R^N$ (a vector that gathers variables describing pointwise the
electrophysiological state of the heart cells). In the monodomain model, the
diffusion operator is $A$($:=\div(g(x)\nabla \cdot)$), and the reaction terms are
the nonlinear functions $f_1$, $f_2$.
These functions model the cellular electrophysiology.
They are called ionic models.
Ionic models are of complex nature,
see e.g. \cite{beeler-reuter, LR2a, tnnp, winslow}. A special attention has to
be paid to the number of evaluations of the functions $f_1$ and $f_2$, and
implicit solvers are usually avoided. Though we ultimately use an
implicit/explicit method to solve the PDE, we need an efficient, fast and robust
method to integrate the reaction terms. Therefore, this article focuses on the
time integration of the stiff ODE system
\begin{equation}
  \label{eq:F1}
  \diff{y}{t}
  = f(t,y), \quad y(0) = y_{0},
\end{equation}
in the special cases where $f(t,y)$ is an ionic model from cellular
electrophysiology.
For that case, stiffness is due to the co existence of fast and slow variables.
Fast variables are given in  \eqref{eq:F1} by  equations  of the form,
\begin{equation}
  \label{eq:F2-0}
  \diff{y_i}{t}
   = f_i(t,y) = a_i(t,y)y_i + b_i(t,y).
\end{equation}
Here $ a_i(t,y)\in\R$ is provided by the model.
This scalar rate of variation will be inserted in the numerical method to stabilize its resolution.
% The main difficulty in the field is the stiffness both in space and in time
% displayed by the solutions. This difficulty is strengthened by the nature of
% the reaction terms $f_1$, $f_2$. These terms are ionic models describing the
% electrical activity at the cellular level. Ionic models for heart cells are
% of
% rather complex nature, . The
% calculations $(v,w,x,t)\rightarrow f_i(v,w,x,t)$ represent a large
% computational load. The total amount of such calculations need to be
% maintained as low as possible by the chosen numerical method. Fully implicit
% time stepping methods (that require a non linear solver) therefore are
% avoided.

% In order to build a numerical method that allows accurate simulations on
% coarse time and space grids, and that avoids non linear solver, we will
% consider high order semi implicit time stepping methods. Such methods are
% implicit regarding the diffusion term $A$ but explicit regarding the reaction
% term. The integration of the diffusive part can be held considering classical
% implicit solvers (for instance BDF$_k$). In contrast the definition of an high
% order stable and explicit solver for the reactive part is quite challenging up
% to date.

Exponential integrators are a class of explicit methods meanwhile exhibiting
strong stability properties. They have motivated many studies along the past 15
years, among which we quote e.g.  \cite{hochbruck-98, Cox_Mat, oster_ex_RK,
  exp-rozenbrock-2009, Tokman-2012, Ostermann-ERK-2014} and refer to
\cite{Bor_W, Hochbruck-Ostermann-2010, hochbruck-2015} for general reviews. They
already have been used in cardiac electrophysiology, as e.g. in
\cite{perego-2009, borgers-2013}. Exponential integrators are based on a
reformulation of \eqref{eq:F1} as,
\begin{equation}
  \label{eq:F2}
  \diff{y}{t}
   = a(t,y)y+b(t,y), \quad y(0) = y_{0},
\end{equation}
(with $f = ay+b$) where the linear part $a(t,y)$ is used to stabilize the
resolution. Basically $a(t,y)$ is assumed to capture the stiffest modes of the
Jacobian matrix of system \eqref{eq:F1}. Stabilization is brought by performing
an exact integration of these modes. This exact integration involves the
computation of the exponential $\exp(a(t_n,y_n)h)$ at the considered point. This
computation is the supplementary cost for exponential integrators as compared to
other time stepping methods.

Exponential integrators of Adams type are explicit multistep exponential
integrators. They were first introduced by Certaine \cite{Cer_J} in 1960 and N{\o}rsettin \cite{Nors_EAB} in 1969 for a constant
linear part $A=a(t,y)$ in~\eqref{eq:F2}. The schemes are derived using a
polynomial interpolation of the non linear term $b(t,y)$. It recently received
an increasing interest~\cite{Tokman-2006, Cal_Pal, Oster_class_Exp} and various
convergence analysis have been completed in this particular case~\cite{EAB_M_O,
  Auzinger-2011, oster-2013}.
Non constant linear parts have been less studied.
Lee and Preiser \cite{Lee_Stan} in 1978 and by Chu \cite{CHU} in 1983
first suggested to rewrite  the equation~\eqref{eq:F1}  at each time instant $t_n$ as,
rewritten as,
\begin{equation}
  \label{F3}
  \diff{y}{t}
  = a_n y + g_n(t,y),\quad y(t_n) = y_n,
\end{equation}
with $a_n = a(t_n, y_n)$ and $g_n(t,y) = b(t,y) + (a(t,y)- a_n) y$. In the
sequel, $a_n$ is referred to as the {\em stabilizer}. It is updated after each
time step. 
Recently, Ostermann \textit{et al} \cite{EAB_M_O,oster-2013}
analyzed 
the linearized exponential Adams method, where the
stabilizer $a_n$ is set to the Jacobian matrix of $f(t,y)$ in~\eqref{eq:F1}.
This choice requires the computation of a matrix exponential at every time step.
Moreover, when the fast variables of the system are known, stabilization can be performed only on these variables. Considering the full Jacobian as the stabilizer
implies unnecessary computational efforts.
To avoid these problems, an alternative is to set the stabilizer as a part or as an approximation of the Jacobian.
This has been analyzed in \cite{Tranquilli-Sandu-2014} and in
\cite{Rainwater-Tokman-2016} for 
exponential Rosenbrock and  exponential Runge Kutta type methods respectively.
That strategy is well adapted to cardiac electrophysiology, where a diagonal stabilizer associated with the fast variable is directly provided by the model with equation (\ref{eq:F2-0}).
The present contribution is to analyze general varying $a(t,y)$ in (\ref{eq:F2}) for
exponential integrators of Adams type, 
referred to as \textit{exponential Adams-Bashforth}, and
shortly denoted EAB. 
Together with the EAB scheme, we introduce a new variant,
that we called \textit{integral exponential Adams-Bashforth}, denoted I-EAB.

% Providing a versatile choice for $a(t,y)$ however has significant practical
% consequences. For instance, the stabilizer $a_n$ may be chosen diagonal, so as
% to avoid the resort to complex Padé or rational approximations for the
% computation of $\exp(a_n h)$ (see~\cite{Hochbruck-Ostermann-2010,
% hochbruck-2015}). Moreover, for complex systems of equations~\eqref{eq:F1}, a
% {\em good choice} for $a(t,y)$ might be deduced from the underlying physics of
% the problem itself. This is particularly true when the stiffness is caused by
% only a part of the variables, so avoiding the superfluous computation of the
% whole Jacobian matrix. This situation typically occurs in the framework of
% electrophysiology, as developed in section \ref{sec:num-res}.

% \todo{Ou alors on le met là ?}

The convergence analysis held in \cite{EAB_M_O} extends to the case of general
varying stabilizers. However there is a lack of results concerning the stability
in that case: for instance, consider the simpler exponential Euler method,
defined by
\begin{displaymath}
  y_{n+1} = s(t_n, y_n,h) := \e^{a_n h} y_n + h \phi_1(a_n h)b_n,\quad
  \phi_1(z) = (\e^z - 1)/z. 
\end{displaymath}
Stability under perturbation (also called 0-stability) can be easily proven
provided that the scheme generator $s(t,y,h)$ is globally Lipschitz in $y$ with
a constant bounded by $1 + Ch$. Therefore stability under perturbation
classically is studied by analyzing the partial derivative $\partial_y s$. This
can be done in the case where $a(t,y)$ is either a constant operator or a
diagonal varying matrix. In the general case however things turn out to be more
complicated. The reason is that we do not have the following general series
expansion, as $\epsilon\rightarrow 0$,
\begin{displaymath}
  e^{M + \epsilon N} \neq e^{M} + \epsilon e^{M} N + O(\epsilon^2),
\end{displaymath}
unless the two matrices $M$ and $N$ are commuting. {\em As a consequence
  differentiating $\e^{a(t,y)h}$ in $y$ cannot be done without very restrictive
  assumptions on $a(t,y)$.}
We present here a stability analysis for general varying stabilizers. 
This will be done by introducing relaxed stability
conditions on the scheme generator $s(t,y,h)$. Together with a consistency
analysis, it provides a new proof for the convergence of the EAB schemes, in
the spirit of~\cite{EAB_M_O}.

Stability under perturbation provides results of qualitative nature. In
addition, the Dahlquist stability analysis strengthens these results. It is a
practical tool that allows to dimension the time step $h$ with respect to the
variations of $f(t,y)$ in equation~\eqref{eq:F1}. The analysis is made by
setting $f(t,y)=\lambda y$ in~\eqref{eq:F1}. For exponential
integrators with general varying stabilizer, the analysis must
incorporate the decomposition of $f(t,y)=\lambda y$ used
in~\eqref{eq:F2}. The stability domain of the considered method will depend on
the relationship between $\lambda$ and $a(t,y)$, following a concept first
introduced in \cite{perego-2009}. We numerically establish that EAB methods are
$A(\alpha)$ stable provided that the stabilizer is sufficiently close to the
system Jacobian matrix (precise definitions are in section
\ref{sec:stab-domain}). Moreover the angle $\alpha$ approaches $\pi/2$ when the
stabilizer goes to the system Jacobian matrix. In contrast, there exists no
$A(0)$ stable explicit linear multistep method (see \cite[chapter
V.2]{Hairer-ODE-II}. This property is  remarkable for explicit methods.

Numerical experiments for the EAB and I-EAB scheme are provided in section
\ref{sec:num-res}, in the context of cardiac electrophysiology. 
% The stabilizer
% is set to a part of the diagonal of the Jacobian matrix that corresponds to the
% model stiffest equations. 
Robustness to stiffness is studied with this
choice. It is numerically shown to be comparable to implicit methods both in
terms of accuracy and of stability condition on the time step. We conclude that
EAB methods are well suited for solving stiff differential problems. In
particular they allow computations at large time step with good accuracy
properties and cheap cost.
\\ \\ \indent
The article is organized as follows. The EAB and I-EAB methods are introduced
in section~\ref{sec:def-meth}. The general stability and convergence results are
stated and proved in section~\ref{sec:ab-theory}. The EAB and I-EAB stability
under perturbation and convergence are proved in
section~\ref{sec:proof-conv}. The Dahlquist stability is investigated in
section~\ref{sec:stab-domain}, and the numerical experiments end the article, in
section \ref{sec:num-res}.

In all this paper,   $h>0$ is a constant time-step and 
$t_n = n h$ are the time instants associated with  the numerical approximate $y_n$ of the solution of the ODE (\ref{eq:F1}).

\section{Scheme definitions}
\label{sec:def-meth}

\subsection{The EAB$_k$ method}
\label{sec:EABk}
The exact solution at time $t_{n+1}$ to the equation~\eqref{F3} (with
$a_n = a(t_n,y_n)$) is given by the variation of the constants formula:
\begin{equation}
  \label{eq:variation-const}
  y(t_{n+1}) = e^{a_nh} \left( y(t_n) +
    \int_{0}^{h} e^{-a_n \tau} g_n(t_n+\tau, y(t_n+\tau)) \d{\tau}
  \right).
\end{equation}
Using the $k$ approximations $y_{n-j} \simeq y(t_{n-j})$ for
$j=0\ldots k-1$, we build the Lagrange polynomial $\tilde{g}_n$ of degree at
most $k-1$ that satisfies,
\begin{equation}
  \label{def:gn}
  \tilde{g}_n(t_{n-j}) = g_{nj} := g(t_{n-j}, y_{n-j}),\quad 0\le j\le k-1.
\end{equation}
It provides the numerical approximation $y_{n+1} \simeq y(t_{n+1})$ as
\begin{equation}
  \label{EAB-1}
  y_{n+1} = \e^{a_nh} \left( y_{n}  + 
    \int_0^h \e^{-a_{n} \tau } \tilde{g}_n(t_n + \tau)\d{\tau}  
  \right).
\end{equation}
The Taylor expansion of the polynomial $\tilde{g}_n$ is
$\tilde{g}_n(t_n + \tau) = \sum_{j=1}^{k} \frac{\gamma_{nj}}{(j-1)!}  \left(\tau
  / h \right)^{j-1}$, where the coefficients $\gamma_{nj}$ are uniquely
determined by \eqref{def:gn}, and actually given in table \ref{tab:gamma-nj} for
$k=1,2,3,4$. An exact integration of the integral in equation~\eqref{EAB-1} may
be performed:% We find that
\begin{equation}
  \label{EAB-2}
  y_{n+1} = \e^{a_n h} y_n + h \sum_{j=1}^{k} \phi_{j}(a_n h)\gamma_{nj},
\end{equation}
where the functions $\phi_j$, originally introduced in \cite{Nors_EAB}, are
recursively defined (for $j\ge 0$) by
\begin{equation}
  \label{phi_k}
  \phi_{0}(z) =  \e^z, \quad \phi_{j+1}(z) =
  \frac{\phi_{j}(z)-\phi_j(0)}{z} \quad  \text{and} \quad \phi_j(0) =
  \frac1{j!}.
\end{equation}
The equation~\eqref{EAB-2} defines the Exponential Adams-Bashforth method
of order $k$, denoted by EAB$_k$.
\begin{table}[h!]
  \caption{Coefficients $\gamma_{nj}$ for the EAB$_k${} schemes}
  \label{tab:gamma-nj}
  \centering
  \begin{tabular}{lllll}
    \toprule
    $k$ & 1 & 2 & 3 & 4 \\
    \midrule
    $\gamma_{n1}$ & $g_n$ & $g_n$ & $g_n$ & $g_n$ \\
    $\gamma_{n2}$ &  & $g_n - g_{n-1}$ & $\frac{3}{2} g_n -2 g_{n-1} + \frac{1}{2} g_{n-2}$ & $\frac{11}{6}g_n-3g_{n-1}+\frac{3}{2}g_{n-2}-\frac{1}{3}g_{n-3}$\\
    $\gamma_{n3}$ &  &  & $g_n-2g_{n-1} + g_{n-2}$ & $2g_n-5g_{n-1}+4g_{n-2}-g_{n-3}$ \\
    $\gamma_{n4}$ &  &  &  & $g_n-3g_{n-1}+3g_{n-2}-g_{n-3}$ \\
    \bottomrule
  \end{tabular}      
\end{table}
\begin{remark}
  When $a(t,y)=\diag(d_i)$ is a diagonal matrix,
  $\phi_k(a_n h) = \diag(\phi_k(d_i))$ can be computed component-wise. Its
  computation is straightforward.
\end{remark}  
\begin{remark}
  With the definition \eqref{phi_k}, the functions $\phi_k$ are analytic on the
  whole complex plane. Therefore the EAB$_k${} scheme definition~\eqref{EAB-2} makes
  sense for a matrix term $a(t,y)$ in equation~\eqref{eq:F2} without particular
  assumption.% It also makes sense for sectorial operators since the
  % $\phi_k$ are bounded on the left half complex plane.
\end{remark}  
\begin{remark}
  The computation of $y_{n+1}$ in the formula~\eqref{EAB-2} requires the
  computation of $\phi_{j}(a_n h)$ for $j=0,\dots,k$. This computational effort
  can be reduced with the recursive definition~\eqref{phi_k}. In practice only
  $\phi_{0}(a_n h)$ needs to be computed. This is detailed in section
  \ref{sec:implementation}.
\end{remark}
\subsection{A variant: the I-EAB$_k$ method}
\label{sec:IEABk}
If the matrix $a(t,y)$ is diagonal, we can take advantage of the following
version for the variation of the constants formula:
\begin{displaymath}
  y(t_{n+1}) = \e^{A_n(h)} \left( y(t_n)+ \int_{0}^{h}
    \e^{-A_n(\tau)} b(y(t_n+\tau),t_n+\tau) \d{\tau} \right),
\end{displaymath}
where $A_n(\tau) = \int_{0}^{\tau} a(t_n+\sigma,y(t_n+\sigma)) d\sigma$.
% 
% where $A_n'(\tau)= a(t_n+\tau,y(t_n+\tau))$.
% 
% Note that this formula is no longer valid in the case of a varying and non
% diagonal term $a(t,y)$.
% 
An attempt to improve the EAB$_k$ formula~\eqref{EAB-2} is to replace $a(t,y)$ and
$b(t,y)$ in the integral above by their Lagrange interpolation polynomials. At
time $t_n$, we define the two polynomials $\tilde{a}_n$ and $\tilde{b}_n$ of
degree at most $k-1$ so that:
\begin{displaymath}
  \tilde{a}_n(t_{n-j}) = a(t_{n-j}, y_{n-j}),\quad \tilde{b}_n(t_{n-j}) =
  b(t_{n-j}, y_{n-j}),\quad j=0\ldots k-1, 
\end{displaymath}
and the primitive
$\tilde{A}_n(\tau) = \int_{0}^{\tau} \tilde{a}(t_n+\sigma,y(t_n+\sigma))
d\sigma$. The resulting approximate solution at time $t_{n+1}$ is finally given
by the formula
\begin{equation}
  \label{sch_IEAB}
  y_{n+1} = \e^{\tilde{A}_n(h)} \left( y_{n} + 
    \int_{0}^{h} \e^{-\tilde{A}_{n}(\tau)} \tilde{b}_n(t_n+\tau) \d{\tau}
  \right).
\end{equation}
The method is denoted I-EAB$_k$ (for Integral EAB$_k$). Dislike for the
formula~\eqref{EAB-1}, no exact integration formula is available, because of the
term $\e^{-\tilde{A}_{n}(\tau)}$. A quadrature rule is required for the actual
numerical computation of the integral in
formula~\eqref{sch_IEAB}. Implementation details are given in
section~\ref{sec:implementation}.
\section{Stability conditions and convergence}
\label{sec:ab-theory}
The equation~\eqref{eq:F1} is considered on a finite dimensional vector space
$E$ with norm $|\cdot|_E$. We fix a final time $T>0$ and assume that equation (\ref{eq:F1}) has
a solution $y$ on $[0,T]$. We adopt the general settings for the analysis
of $k$-multistep methods following \cite{Hairer-ODE-I}. The space $E^k$ is
equipped with the maximum norm $|Y|_\infty = \max_{1\le i\le k} |y_i |_E$ with
$Y=(y_1,\dots,y_k)\in E^k$. A $k$-multistep scheme is defined by a mapping,
\begin{displaymath}
  s:~ (t,Y,h) \in \R\times E^k \times\R^+ \mapsto s(t,Y,h)\in E.
\end{displaymath}
For instance, the EAB$_k$ scheme rewrites as $y_{n+1}=s(t_n, Y, h)$ with
$Y=(y_{n-k+1}, \dots, y_n)$ in the formula~\eqref{EAB-2}. The scheme generator
is the mapping $S$ given by
\begin{displaymath}
  S~: (t,Y,h) \in \R\times E^k \times\R^+ \longrightarrow 
  \left( y_2 , \ldots , y_k, s(t,Y,h)\right) \in    E^k .
\end{displaymath}
A numerical solution is a
sequence $(Y_n)$ in $E^k$ for  $n\ge k-1$
so that
\begin{equation}
  \label{num-sol}
  Y_{n+1} = S(t_n, Y_n, h)\quad\text{for}\quad n\ge k-1,
\end{equation}
$Y_{k-1} = (y_0,\dots,y_{k-1})$ being its 
initial condition.
A perturbed numerical solution is a sequence $(Z_n)$ in $E^k$ for
$n\ge k-1$
such that
% $k-1\le n\le T/h$ such that
\begin{equation}\label{num-sol-pert}
  Z_{k-1} = Y_{k-1} + \xi_{k-1}, \quad Z_{n+1}= S(t_n,Z_{n},h) +
  \xi_{n+1} \quad\text{for}\quad n\ge k-1,%k\le n+1 \le T/h,
\end{equation}
with $(\xi_n)\in E^k$ for $n\ge k-1$.
The scheme is said to be stable under perturbation (or 0-stable) if: being given
a numerical solution $(Y_n)$ as in \eqref{num-sol} there exists a (stability)
constant $L_s>0$ so that for any perturbation $(Z_n)$ as defined in
\eqref{num-sol-pert} we have,
\begin{equation}
  \label{def:0-stablity}
  \max_{k-1\le n\le T/h} |Y_n-Z_n|_\infty  \le L_s 
  \sum_{k-1\le n\le T/h} |\xi_n|_\infty .
\end{equation}
\begin{proposition}
  \label{prop:stab}
  Assume that there exists constants $C_1>0$ and $C_2>0$ such that
  \begin{gather}
    \label{eq:cond-stab-2}
    1 + |S(t,Y,h) |_\infty \le \left( 1 + |Y |_\infty \right) (1+C_1 h), \\
    \label{eq:cond-stab-1}
    |S(t,Y,h) - S(t,Z,h) |_\infty \le |Y - Z |_\infty \left( 1 + C_2 h ( 1 + |Y
      |_\infty ) \right),
  \end{gather}
  for $0\le t\le T$, and for $Y,Z\in E^k$. Then, the numerical scheme is stable
  under perturbation with the constant $L_s$ in~\eqref{def:0-stablity} given by
  \begin{equation}
    \label{def-Ls}
    L_s = \e^{C^\star T},\quad C^\star := C_2\e^{C_1T}\left( 1 + |Y_{k-1} |_\infty  \right).
  \end{equation}
\end{proposition}
\begin{proof}
  Consider a numerical solution $(Y_n)$ in \eqref{num-sol}. A recursion on
  condition \eqref{eq:cond-stab-2} gives,
  \begin{displaymath}
    1 + |Y_n |_\infty \le \left(1 + |Y_{k-1}
      |_\infty\right) (1+C_1 h)^{n-k+1} \le \e^{C_1 T} \left(1 + |Y_{k-1}
      |_\infty\right),
  \end{displaymath}
  since $(1+x)^p \le \e^{px}$ and $(n-k+1) h \le nh \le T$.
  Now, consider a perturbation $(Z_n)$ of $(Y_n)$ given by
  \eqref{num-sol-pert}. Using the condition \eqref{eq:cond-stab-1} together with
  the previous inequality,
  \begin{align*}
    |Y_{n+1} - Z_{n+1}|_\infty &\le |S(t_n,Y_n,h) - S(t_n,Z_n,h) |_\infty +
    |\xi_{n+1} |_\infty \\
    &\le |Y_n - Z_n |_\infty \left( 1 + C_2 h ( 1 + |Y_n |_\infty ) \right) +
    |\xi_{n+1} |_\infty \\ 
    &\le |Y_n - Z_n |_\infty \left( 1 + C_2 \e^{C_1 T} ( 1 + |Y_{k-1} |_\infty )h
    \right) + |\xi_{n+1} |_\infty  \\
    &\le |Y_n - Z_n |_\infty \left( 1 + C^\star h \right) + |\xi_{n+1}|_\infty,
  \end{align*}
  where $C^\star := C_2\e^{C_1T}\left( 1 + |Y_{k-1} |_\infty \right)$. By
  recursion we get,
  \begin{align*}
    |Y_n - Z_n |_\infty &\le \left(1+C^\star h\right)^{n-k+1} |Y_{k-1} -
    Z_{k-1}|_\infty + \sum_{i=0}^{n-k} \left(1+C^\star h\right)^i
    |\xi_{n-i}|_\infty \\ 
    &\le \left (1+C^\star h\right)^n \sum_{i=k-1}^{n} |\xi_{i}|_\infty \le
    \e^{c^\star T} \sum_{i=k-1}^{n} |\xi_{i}|_\infty, 
  \end{align*}
  which ends the proof.
\end{proof}
Like in the classical cases, stability under perturbation together with
consistency ensures convergence. 
Let us specify this point.
For the considered solution $y(t)$ of problem \eqref{eq:F1} on $[0,T]$, we define,
\begin{equation}
  \label{def:Y(t)}
  Y(t)=\left (y\left(t - (k-1)h\right),\ldots y(t)\right )\in E^k\quad 
  \text{for}\quad 
  (k-1)h\le t\le T.
\end{equation}
The local error at time $t_n$ is,
\begin{equation}
  \label{def:local-error}
  \epsilon(t_n,h) = Y(t_{n+1}) - S(t_n, Y(t_n),h).
\end{equation}
The scheme is said to be consistent of order $p$ if there exists a (consistency)
constant~$L_c>0$ only depending on $y(t)$ such that
\begin{displaymath}%\label{eq:consistant}
  \max_{k-1 \le n \le T/h}  |\epsilon(t_n,h)|_\infty \le L_c h^{p+1}.
\end{displaymath}

\begin{corollary}
  \label{cor:conv}
  If the scheme satisfies the stability conditions~\eqref{eq:cond-stab-2}
  and~\eqref{eq:cond-stab-1}, and is consistent of order $p$, then a numerical
  solution $(Y_n)$ given by \eqref{num-sol} satisfies,
  \begin{equation}
    \label{def:conv}
    \max_{k-1 \le n \le T/h} |Y(t_{n}) - Y_n |_\infty \le L_s L_c T h^{p} + L_s
    |\xi_0|_\infty,
  \end{equation}
  where $\xi_0 = Y(t_{k-1}) - Y_{k-1}$ denotes the error on the initial
  data, and the constant $L_s$ is as in equation~\eqref{def-Ls}.
\end{corollary}
\begin{remark}
  Note that the stability constant $L_s$ in \eqref{def-Ls} depends on
  $|Y_{k-1}|_\infty$, and then on $h$. This is not a problem since $L_s$ can be
  bounded uniformly as $h\rightarrow 0$ for $Y_{k-1}$ in a neighbourhood of
  $y_0$.
\end{remark}
\begin{proof}
  We have $Y(t_{k-1}) = Y_{k-1} + \xi_0$ and
  $Y(t_{n+1}) = S(t_n, Y(t_n),h) + \epsilon(t_n,h)$. Therefore the sequence
  $(Y(t_n))$ is a perturbation of the numerical solution $(Y_n)$ in the sense
  of~\eqref{num-sol-pert}. As a consequence, proposition \ref{prop:stab} shows
  that
  \begin{align*}
    \max_{k-1\le n\le T/h} |Y_n-Y(t_n)|_E \le L_s \left( |\xi_0| + \sum_{k\le
        n\le T/h} |\epsilon(t_n,h)| \right) \le L_s |\xi_0| + L_s L_c \left(
      \sum_{k\le n\le T/h} h \right) h^{p},
  \end{align*}
  and the convergence result follows.
\end{proof}
\section{EAB$_k$ and I-EAB$_k$ schemes analysis}
\label{sec:proof-conv}
The space $E$ is assumed to be $E = \R^N$ with its canonical basis and with
${|\cdot|_E}$ the maximum norm. The space of operators on $E$ is equipped with
the associated operator norm, and associated to $N\times N$ matrices. Thus
$a(t,y)$ is a $N\times N$ matrix and its norm $|a(t,y)|$ is the matrix norm
associated to the maximum norm on $\R^N$.
\\
It is commonly assumed for the numerical analysis of ODE solvers that $f$ in the
equation~\eqref{eq:F1} is uniformly Lipschitz in its second component $y$. With the formulation (\ref{eq:F2}), the following supplementary hypothesis  will be needed: on $\R\times E$,
\begin{equation}
  \label{eq:a-b-f-Lipschitz}
  \vert a(t,y)\vert \le M_a,\quad 
  a(t,y),~b(t,y)~~\text{and}~f(t,y)~~\text{uniformly Lipschitz in}~y.
\end{equation}
We denote by $K_f$, $K_a$ and $K_b$ the Lipschitz constant for $f$, $a$ and $b$ respectively.
\begin{theorem}
  \label{thm:conv}
  With the assumptions \eqref{eq:a-b-f-Lipschitz}, the EAB$_k$ and
  I-EAB$_k$ schemes are stable under perturbations. Moreover, if $a$ and $b$ are
  $\mathcal{C}^k$ regular on $\R\times E$,
  then the EAB$_k$ and I-EAB$_k$ schemes are consistent of order $k$. Therefore they
  converge with order $k$ in the sense of inequality~\eqref{def:conv}, by
  applying corollary~\ref{cor:conv}.
\end{theorem}
The stability and consistency are proved in sections~\ref{sec:eabk-stability}
and \ref{sec:eabk-consistency}, respectively. Preliminary tools and definitions
are provided in the sections \ref{sec:interp} and \ref{sec:sch-generator}.
\subsection{Interpolation results}
\label{sec:interp}
Consider a function $x:~\R\times E \longrightarrow \R$ and a triplet $(t,Y,h)\in \R\times E^k\times \R^+$ with $Y=(y_1,\dots,y_k)$. 
We set to $\xt$ the polynomial with degree less than $k-1$ so that
\begin{displaymath}
  \xt (t-ih) =     
  x(t-ih,y_{k-i}),\quad
  0\le i\le k-1.
\end{displaymath}
We then extend component-wise this definition to vector valued or matrix valued functions $x$ (\textit{e.g.} the functions $a$, $b$ or $f$).
\begin{lemma}
  \label{lem:pol-interp}
  There exists an (interpolation) constant $L_i>0$ such that, for any 
  function $x:\R\times E \mapsto \R$, 
  \begin{align}
    \label{interp-norm}
    \sup_{t\le \tau \le t+h}
    |\xt(\tau)| &\le L_i 
    \max_{0\le i\le k-1} \left| x(t-ih,y_{k-i}) \right|,
    \\
    \label{interp-norm-2}
    \sup_{t\le \tau \le t+h}
    |\tilde{x}_{_{[t,Y_1,h]}}
    -\tilde{x}_{_{[t,Y_2,h]}}
    | &\le L_i 
    \max_{0\le i\le k-1} \left| x(t-ih,y_{1,k-i}) 
      - x(t-ih,y_{2,k-i}) 
    \right|,
  \end{align}
  Consider a function $y:~[0,T]\rightarrow E$ and assume that $x$ and $y$ have a $\mathcal{C}^k$ regularity. Then, when $[t-(k-1)h,t+h]\subset [0,T]$,
  \begin{equation}
    \label{interp-error}
    \sup_{t\le \tau\le t+h} \left|x(t,y(t)) -
      \tilde{x}_{_{[t, Y(t),h]}}\right|_E \le \sup_{[0,T]} \left|
      \diff[k]{f}{t}\left(f(t,y(t))\right) \right| h^k,
  \end{equation}
  with $Y(t)$ defined in (\ref{def:Y(t)}).
  \\
  For a vector valued function in $\R^d$ the previous inequalities hold when considering the max norm on $\R^d$.
  For a  matrix valued function  in $\R^d\times \R^d$ this is also true for the operator norm on $\R^d\times \R^d$ when multiplying the constants in the inequalities (\ref{interp-norm}), (\ref{interp-norm-2}) and (\ref{interp-error}) by $d$.
\end{lemma}
\begin{proof} 
  The space $\Pk$ of the polynomials $p$ with
  degree less than $k-1$ is equipped  with the norm $ \sup_{[0,1]} |p(\tau)|$.
  On $\R^k$ is considered the max norm.
  To $R=(r_1,\ldots,r_k)\in \R^k$ is associated $\L R\in\Pk$ uniquely determined by $\L R(-i)= r_{k-1}$ for $i=0\ldots k-1$.
  The mapping $\L$ is linear. 
  Let $C_{\L}$ be its continuity constant (it only depends on $k$). 
  \\
  We fix $x:~\R\times  E\rightarrow \R$ and $(t,h)\in \R\times \R^+$.
  Consider the vector $Y=(y_1,\ldots,y_k)\in E^k$ and set the vector 
  $R\in\R^k$ by $R=\left (x(t-(k-1)h, y_1),\ldots,x(t,y_k)\right)$.
  We have $\xt(t+\tau) = \L R(\tau / h)$.
  The relation (\ref{interp-norm}) exactly is the continuity of 
  $\L$ and $L_i=C_{\L}$.
  \\
  Consider $Y_1$, $Y_2\in E^k$ and the associated vectors $R_1$, $R_2$ as above.
  We have 
  $(x_{_{[t,Y_1,h]}} - x_{_{[t,Y_2,h]}})(t+\tau) = \L\left (R_1-R_2\right )(\tau / h)$. Again, relation (\ref{interp-norm-2}) comes from the continuity of $\L$.
  \\
  Let $\phi:~\R\rightarrow \R$  be a $\mathcal{C}^k$ function, 
  its interpolation polynomial $\tilde{\phi}$ at the points $t-(k-1)h,\ldots, t$ is considered.
  A classical result on Lagrange interpolation applied to $\phi$ states that, 
  for all $\tau\in(t,t+h)$, there exists
  $\xi \in (t-(k-1)h,t+h)$, such that
  $\left (\phi - \tilde{\phi}\right )(\tau) = \frac1{k!}
  \phi^{(k)}(\xi) \pi(\tau)$,
  where $\pi(\tau) = \prod_{i=1}^{k} (\tau-t_{i})$. 
  For
  $\tau\in(t,t+h)$, we have $|\pi(\tau)| \le k!\, h^k$.
  This proves (\ref{interp-error}) by setting $\phi(t) = x(t,y(t))$.
  \\
  For a vector valued function $x:~\R\times  E\rightarrow \R^d$, these three inequalities holds by processing component-wise and when considering the max norm on $\R^d$.
  \\
  For a matrix valued function $x:~\R\times  E\rightarrow \R^d\times \R^d$, the extension is direct when considering the max norm $\vert \cdot\vert _\infty$ on $\R^d\times \R^d$ (\textit{i.e.} the max norm on the matrix entries). 
  The operator norm $\vert \cdot\vert$ is retrieved with the inequality $\vert \cdot\vert _\infty \le d \vert \cdot\vert$   
\end{proof}
\subsection{Scheme generators}
\label{sec:sch-generator}
Let us consider $(t,Y,h)\in \R\times E^k\times \R^+$ with $Y=(y_1,\dots,y_k)$.
With the notations used in the previous subsection, we introduce the interpolations 
$\at$ and $\bt$
for the functions $a$ and $b$ in \eqref{eq:F2}. Thanks to the definition
\eqref{sch_IEAB}, the I-EAB$_k$ scheme generator is defined by
\begin{equation}
  \label{IEABk-generator}
  s(t,Y,h) = z(t+h)
  \quad \text{with}\quad 
  \diff{z}{\tau} 
  = \at(\tau) z(\tau) + \bt(\tau),\quad z(t) = y_k,
\end{equation}  
We introduce the polynomial $\gt$ with degree less than $k-1$ that satisfies,
\begin{displaymath}
  \gt (t-ih)= 
  f(t-ih, y_{k-i}) - a(t,y_k)y_{k-i},\quad   i=0\ldots k-1.
\end{displaymath}
The function $\tilde{g}_n$ in \eqref{def:gn} is given by
$\tilde{g}_n = \bar{g}_{_{[t_n,Y_n,h]}}$ with $Y_n=(y_{n-k+1}\dots,y_n)$.  With
the definition \eqref{EAB-1}, the EAB$_k$ scheme generator is defined by
\begin{equation}
  \label{EABk-generator}
  s(t,Y,h) = z(t+h)
  \quad \text{with}\quad 
  \diff{z}{\tau} 
  = a(t,y_k) z(\tau) + \gt(\tau),\quad z(t) = y_k,
\end{equation}
We will use the fact that $\gt$ is the interpolator of the function
$\gty:~(\tau,\xi) \rightarrow f(\tau,\xi) - a(t,y_k)\xi$, precisely:
\begin{displaymath}
  \gt=\widetilde{g_{_{t,y_k}}}_{_{[t,Y,h]}}.
\end{displaymath}

These scheme generator definitions will allow us to use the following Gronwall's
inequality (see \cite[Lemma 196, p.150]{gronwall-monograph}).
\begin{lemma}%\label{lem:gronwall}
  Suppose that $z(t)$ is a $C^1$ function on $E$. If there exist $\alpha>0$ and
  $\beta>0$ such that $|z'(t)|_E \le \alpha |t|+ \beta$ for all
  $t\in[t_0,t_0+h]$, then:
  \begin{equation}
    \label{lem:gronwall}
    |z(t)|_E \le |z(t_0)|_E \e^{\alpha h} + \beta h \e^{\alpha h}
    \quad \text{for} \quad 
    t\in[t_0,t_0+h].
  \end{equation}
\end{lemma}
\subsection{Stability}
\label{sec:eabk-stability}
With proposition \ref{prop:stab} we have to prove the stability conditions
\eqref{eq:cond-stab-2} and \eqref{eq:cond-stab-1}. Note that it is sufficient to
prove these relations for $h\le h_0$ for some constant $h_0>0$ since the limit
$h\rightarrow 0$ is of interest here.
\subsubsection{Case of the I-EAB$_k$ scheme}
Consider $(t,h)\in \R \times \R^+$ and a vector $Y=(y_1,\dots,y_k)\in E^k$. We simply
denote $\tilde{a}= \at$ and $\tilde{b}= \bt$. The scheme generator is given by
\eqref{IEABk-generator}. We first have to bound $z(t+h)$ where $z$ is given by
\begin{displaymath}
  z' = \tilde{a}z + \tilde{b},\quad  z(t)=y_k.
\end{displaymath}
On the first hand, with the interpolation bound \eqref{interp-norm},
\begin{displaymath}
  \sup_{t\le \tau \le t+h} | \tilde{a}(\tau) |
  \le L_i \max_{0\le i\le k-1}
  |a(t-ih, y_{k-1})|\le L_i M_a.
\end{displaymath}
On the second hand, the function $b(t,y)$ is globally Lipschitz in $y$ and thus
can be bounded as
$|b(t,y)|_E \le |b(t,0)|_E + K_b|y |_E \le
R_b(|y|_E+1)$, for $0\le t\le T$ and for some constant $R_b$ only depending on
$K_b$ and on $T$. Then with the bound \eqref{interp-norm},
\begin{displaymath}
  \sup_{t\le \tau \le t+h} | \tilde{b}(\tau) |_E
  \le L_i \max_{0\le i\le k-1}
  R_b\left (|y_{k-i}|_E +1\right )\le L_i R_b(|Y|_\infty +1).
\end{displaymath}
By applying the Gronwall inequality \eqref{lem:gronwall}, for $0\le \tau\le h$,
\begin{displaymath}
  |z(t+\tau)|_E \le
  \e^{L_i M_a h}\left (
    |y_k|_E + h L_i R_b(|Y|_\infty +1)
  \right ).
\end{displaymath}
Thus, there exists a constant $C_1$ only depending on $L_i$, $M_a$, $R_b$ and
$h_0$ so that, for $0\le \tau\le h$ and for $0\le h\le h_0$,
\begin{equation}
  \label{eq:up-bound-z}
  |z(t+\tau)|_E \le C_1h + |Y|_\infty (1+C_1h).
\end{equation}
This gives the condition \eqref{eq:cond-stab-2} taking $\tau = h$.

For $j$=1, 2 We consider $Y_j=(y_{j,1},\ldots y_{j,k})\in E^k$ and denote
$\tilde{a}_j=\tilde{a}_{_{[t,Y_j,h]}}$ and
$\tilde{b}_j=\tilde{b}_{_{[t,Y_j,h]}}$ the
interpolations of the functions $a$ and $b$. 
With the definition~\eqref{IEABk-generator} of the I-EAB$_k$ scheme
we have:
$\left |s(t,Y_1,h)-s(t,Y_2,h)\right |_E = |\delta(t+h)|$ with
$\delta = z_1 - z_2$ and with $z_j$ given by
\begin{displaymath}
  z_j' = \tilde{a_j}z_j + \tilde{b_j},\quad  z_j(t)=y_{j,k}.
\end{displaymath}
We then have,
\begin{displaymath}
  \delta' = \tilde{a}_1 \delta + r
  ,\quad 
  r:=(\tilde{a}_1-\tilde{a}_2) z_2 + (\tilde{b}_1-\tilde{b}_2)  .
\end{displaymath}
Using that $a$ and $b$ are Lipschitz in $y$ and with the interpolation bound \eqref{interp-norm-2},
\begin{gather*}
  \sup_{t\le \tau \le t+h} |\tilde{b}_1(\tau)-\tilde{b}_2(\tau) |_E
  \le L_i K_b |Y_1-Y_2|_\infty,\\
  \sup_{t\le \tau \le t+h} |\tilde{a}_1(\tau)-\tilde{a}_2(\tau) |\le
  L_i K_a |Y_1-Y_2|_\infty.
\end{gather*}
With the upper bound \eqref{eq:up-bound-z}, for $t\le \tau\le t+h\le T$ and for
$h\le h_0$,
\begin{align}
  \notag
  | r(\tau) |_E
  &\le
  L_i | Y_1-Y_2|_\infty 
  \left (
    K_b + K_a \left( C_1h + |Y_2|_\infty (1+C_1h)\right )
  \right )
  \\     \label{bound-r(t)}
  &\le
  C | Y_1-Y_2|_\infty \left ( 1 + |Y_2|_\infty \right )
  ,
\end{align}
For a constant $C$ only depending on $h_0$, $K_a$, $K_b$, $L_i$ and $C_1$. We
finally apply the Gronwall inequality \eqref{lem:gronwall}. It yields,
\begin{align*}
  |\delta(t+h)|&\le \e^{L_i M_a h}\left(
    |y_{1,k}-y_{2,k}|_E
    +
    C h |Y_1 - Y_2|_\infty  \left(
      1 + |Y_2|_\infty 
    \right)
  \right)
  \\&\le
  |Y_1 - Y_2|_\infty \e^{L_i M_a h}
  \left( 1 + Ch
    \left(
      1 + |Y_2|_\infty 
    \right) 
  \right),
\end{align*}
This implies the second stability condition \eqref{eq:cond-stab-1} for
$h\le h_0$.
\subsubsection{Case of the EAB$_k$ scheme}
Consider $(t,h)\in \R \times \R^+$ and a vector $Y=(y_1,\dots,y_k)\in E^k$. 
Following the definition of the EAB$_k${} scheme given in
section~\ref{sec:EABk}, we denote $\bar{a}=a(t,y_k)$, $g$ the
function $g(\tau,\xi)=f(\tau,\xi)-\bar{a}\xi$ and $\bar{g}= \gt$. 
We have that $\bar{g}$ is the
Lagrange interpolation polynomial of $g$, specifically
$\bar{g} = \tilde{g}_{_{[t,Y,h]}}$.
The
scheme generator is then given by equation \eqref{EABk-generator}: $s(t,Y,h)=z(t+h)$ with,
\begin{displaymath}
  z' = \bar{a}z + \bar{g},\quad  z(t)=y_k.
\end{displaymath}
We first have the bound $|\bar{a}|\le M_a$. As in the previous
subsection, $f$ being globally Lipschitz in $y$, one can find a constant $R_f$
so that for $0\le t\le T$, $|f(t,y)|_E \le R_f(1+|y|_E)$. It follows
that $|g(\tau,\xi)|_E\le R_f(|y|_E+1) + M_a|y|_E \le C_0/L_i(|y|_E+1)$, with
$C_0/L_i=R_f+M_a$. Therefore, with the interpolation bound \eqref{interp-norm}:
\begin{displaymath}
  \sup_{t\le \tau\le t+h} |\bar{g}(\tau)|_E \le
  L_i \max_{0\le i\le k-1}
  |g(t-ih, y_{k-i})|_E \le C_0(|y|_E+1).
\end{displaymath}
By applying the Gronwall inequality \eqref{lem:gronwall}, for $0\le \tau\le h$,
\begin{displaymath}
  |z(t+\tau)|_E \le
  \e^{M_a h}\left (
    |y_k|_E + h C(|Y|_\infty +1)
  \right ).
\end{displaymath}
Thus, there exists a constant $C_1$ only depending on $M_a$ and $C_0$ so that,
for $0\le \tau\le h$ and for $0\le h\le h_0$, \eqref{eq:up-bound-z} holds.  This
gives the condition \eqref{eq:cond-stab-2}.

We now consider $Y_1$, $Y_2\in E^k$ for $j=$1, 2 and denote as previously,
$\bar{a}_j=a(t,y_{j,k})$, $g_j$ the function
$g_j(\tau,\xi)=f(\tau,\xi)-\bar{a}_j\xi$ and
$\bar{g}_j= \bar{g}_{_{[t,Y_j,h]}}$. With \eqref{EABk-generator},
$|s(t,Y_1,h)-s(t,Y_2,h)|_E = |\delta(t+h)|$ with $\delta = z_1 - z_2$ and with
$z_j$ given by
\begin{displaymath}
  z_j' = \bar{a}_j z_j + \bar{g_j},\quad  z_j(t)=y_{j,k}.
\end{displaymath}
The function $\delta$ satisfies the ODE
\begin{displaymath}
  \delta' = \bar{a}_1 \delta + r(t),\quad r(t):=
  (\bar{a}_1-\bar{a}_2) z_2 + (\bar{g}_1-\bar{g}_2)  .
\end{displaymath}
We have
\begin{align*}
  |g_1(\tau, y_{1,i}) - g_2(\tau, y_{2,i})|_E 
  &\le |f(\tau, y_{1,i}) - f(\tau,
  y_{2,i})|_E + |\bar{a}_1||y_{1,i}-y_{2,i}|_E + |\bar{a}_1 -
  \bar{a}_2||y_{2,i}|_E 
  \\ 
  &\le |Y_1-Y_2|_\infty \left ( K_f + M_a +
    K_a|Y_2|_\infty \right ).
\end{align*}
Thus, with equation \eqref{interp-norm-2}, for some $C>0$,
\begin{align*}
  \sup_{t\le \tau\le t+h} |\bar{g}_1(\tau)- \bar{g}_2(\tau)|_E &\le
  L_i \max_{0\le i\le k-1}
  |g_1(t-ih, y_{1,k-i}) - g_2(t-ih, y_{2,k-i})|_E 
  \\ &\le
  C |Y_1-Y_2|_\infty \left (
    1 + |Y_2|_\infty 
  \right )
  .
\end{align*}
Meanwhile we have the upper bound \eqref{eq:up-bound-z} that gives, for
$t\le\tau\le t+h\le T$ and $h\le h_0$,
\begin{align*}
  \left |(\bar{a}_1-\bar{a}_2) z_2 \right |_E
  &\le 
  M_a \left | Y_1-Y_2\right |_\infty 
  | z_2(\tau)|_E 
  \\&\le
  M_a \left | Y_1-Y_2\right |_\infty 
  \left (C_1h + |Y_2|_\infty (1+C_1h)\right ).
\end{align*}
Altogether, we retrieve the upper bound \eqref{bound-r(t)} on $r(t)$. We can end
the proof as for the I-EAB$_k$ case and conclude that the stability condition
\eqref{eq:cond-stab-1} holds for the EAB$_k$ scheme.
\subsection{Consistency}
\label{sec:eabk-consistency}
A solution $y(t)$ to problem \eqref{eq:F1} on $[0,T]$ is fixed. The functions
$a$ and $b$ in \eqref{eq:F2} are assumed to be $\mathcal{}^k$ regular so that $y$ is
$\mathcal{C}^{k+1}$ regular.
\subsubsection{Case of the EAB$_k$ scheme}
The local error \eqref{def:local-error} for the EAB$_k$ scheme has been analyzed in
\cite{EAB_M_O}. That analysis remains valid for the case presented here and we
only briefly recall it. The local error is obtained by subtracting \eqref{EAB-1}
to \eqref{eq:variation-const}:
\begin{align*}
  |\epsilon(t_n,h)|_E
  &\le \int_0^h
  \e^{M_a (h-\tau)} \left |
    g_n(t+\tau,y(t+\tau)) - \tilde{g}_n(t+\tau)
  \right |_E \d{\tau}
  \\&\le
  h\phi_1(M_a h) h^{k}
  \sup_{[0,T]} \left|
    \diff[k]{}{t}\left(g_n(t,y(t))\right) \right|,
\end{align*}
thanks to the interpolation error estimate \eqref{interp-error}. Finally, with
the upper bound $M_a$ on $a_n$, the last term can be bounded independently on
$n$ for $h\le h_0$.
\subsubsection{Case of the I-EAB$_k$ scheme}
We denote $\tilde{a}=\tilde{a}_{_{[t_n,Y(t_n),h]}}$ and
$\tilde{b}=\tilde{b}_{_{[t_n,Y(t_n),h]}}$. The local error
\eqref{def:local-error} for the I-EAB$_k$ scheme satisfies
$\epsilon(t_n,h) = |\delta(t_{n+1})|_E$ with $\delta = y - z$ and
where $z$ is defined by
\begin{displaymath}
  z' = \tilde{a} z + \tilde{b},\quad z(t_n)=y(t_n),
\end{displaymath}
so that with \eqref{IEABk-generator} we have $s(t_n,Y(t_n),h)=z(t_{n+1})$.  The
function $\delta$ is defined with $\delta(t_n)=0$ and,
\begin{displaymath}
  \delta' = \tilde{a} \delta + r
  ,\quad 
  r(\tau):=(a(\tau,y(\tau))-\tilde{a}(\tau)) y(\tau) 
  +        (b(\tau,y(\tau))-\tilde{b}(\tau))  
  .
\end{displaymath}
The following constants only depend on the considered exact solution $y$, on the
functions $a$ and $b$ in problem \eqref{eq:F2} and on $T$,
\begin{displaymath}
  C_y=\sup_{[0,T]}|y|_E,\quad 
  C_{a,y} = \sup_{[0,T]} \left |
    \diff[k]{}{t} a(t,y(t))
  \right |,\quad 
  C_{b,y} = \sup_{[0,T]} \left |
    \diff[k]{}{t} b(t,y(t))
  \right |.
\end{displaymath}
With the interpolation bound  \eqref{interp-error}, $|r(\tau)|_E \le C h^k$ on
$[t_n,t_{n+1}]$ with $C=C_{a,y}C_y+C_{b,y}$. It has already been showed in
section \ref{sec:eabk-stability} that
$\sup_{[t_n,t_{n+1}]} |\tilde{a}(\tau)|\le L_i M_a$. Therefore, with
the Gronwall inequality \eqref{lem:gronwall},
\begin{displaymath}
  \epsilon(t_n,h) = 
  | \delta(t_{n+1})|_E
  \le
  \e^{L_i M_a h}  h C h^K.
\end{displaymath}
Thus the EAB$_k$ scheme is consistent of order $k$.
\section{Dahlquist stability}
\label{sec:stab-domain}
\subsection{Background}
\begin{figure}[h!]
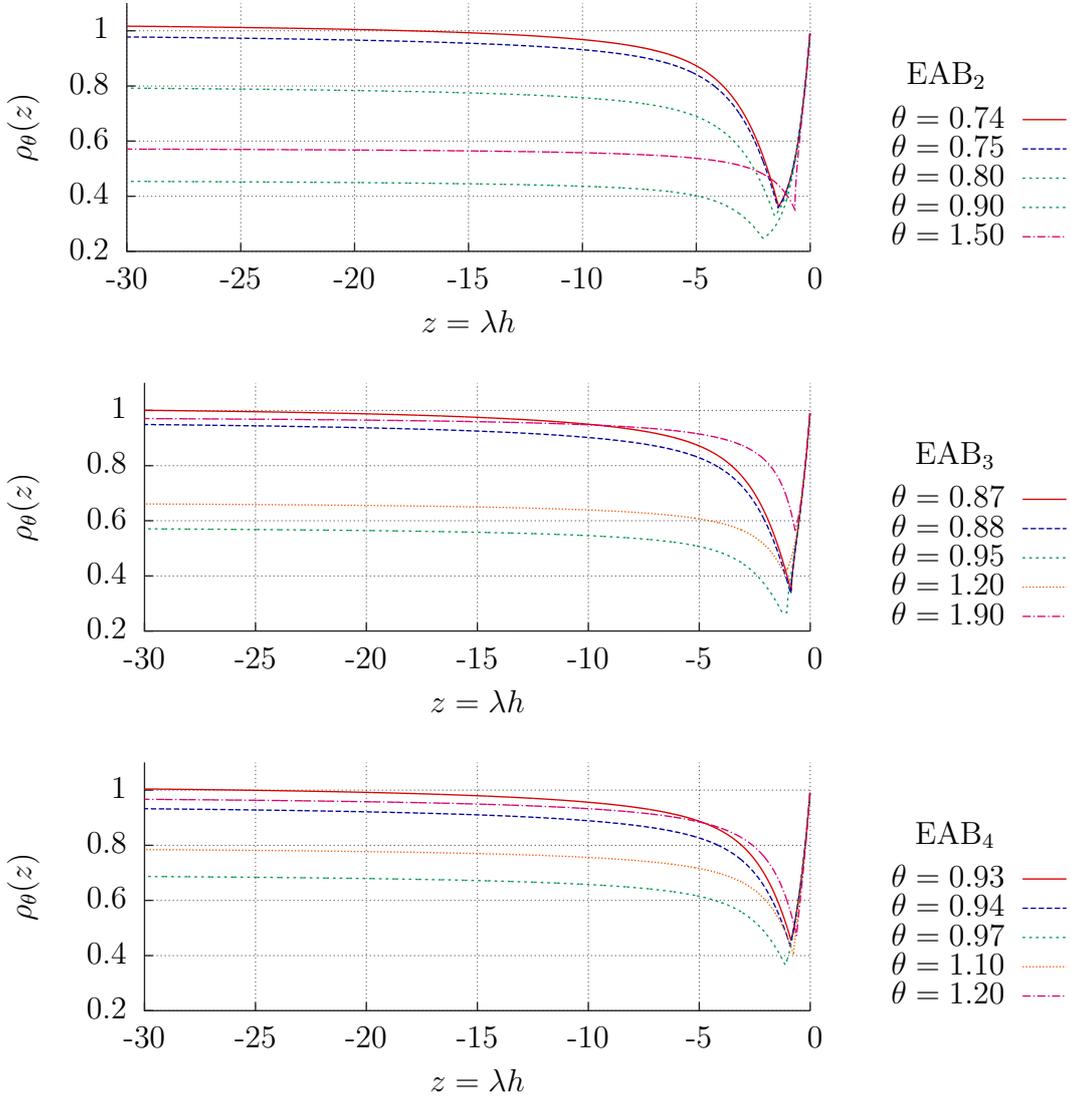

  \centering
  \input{A0_EAB2.tex} \\
  \input{A0_EAB3.tex} \\
  \input{A0_EAB4.tex}
  \caption{
    Stability function $\rho_\theta(z)$ for $z \in \R^-$, for various values of $\theta$ 
    and for the three schemes EAB$_2$, EAB$_3$ and EAB$_4$.
  }
  \label{fig:A0-stab}
\end{figure}
The classical framework for the Dahlquist stability analysis is to set $f(t,y)$
in problem \eqref{eq:F1} to $f=\lambda y$. For linear multistep methods, see
e.g. \cite{Hairer-ODE-II}, the numerical solutions satisfy
$ \left|y_{n+1} / y_n \right|\le \rho(\lambda h), $ where
$\rho:~ \C \rightarrow \R^+$.
The function $\rho$ is the stability function. 
Its definition is detailed below.
 The stability domain is defined by
$ D = \left\{ z\in\C,\quad \rho(z) <1 \right\} $. The scheme is said to be:
\begin{itemize}
\item $A$ stable if $\C^- \subset D$,
\item $A(\alpha)$ stable if $D$ contains the cone with axis $\R^-$ and with half
  angle $\alpha$,
\item $A(0)$ stable if $\R^- \subset D$,
\item stiff stable if $D$ contains a half plane $\text{Re}~z<x\in \R^-$.
\end{itemize}
For exponential integrators, when setting $a(t,y)=\lambda$ in the reformulation
\eqref{eq:F2} of problem \eqref{eq:F1}, the scheme is exact, and therefore also $A$
stable. Such an equality does not hold in general. Then for exponential
integrators the Dahlquist stability analysis has to incorporate the relationship
between the stabilization term $a(t,y)$ in \eqref{eq:F2} and the test function
$f=\lambda y$. This is done here by considering the splitting,
\begin{equation}\label{eq:split-f}
  f=\lambda y = ay + b,\quad 
  a = \theta \lambda \quad {\rm and}\quad 
  b = \lambda(1-\theta)y,
\end{equation}
The parameter $\theta>0$ controls with what accuracy the exact linear part of
$f$ in equation \ref{eq:F1} is captured by $a$ in equation \ref{eq:F2}. In
practice $\theta\ne 1$, though we may hope that $\theta-1$ is small.
In that framework, the stability function and the stability domain depend on
$\theta$, following the idea of Perego and Venezziani in \cite{perego-2009}
(see the remark below).
For a fixed $\theta$, the stability function is $\rho_\theta$ so that
\begin{displaymath}
  \left|\frac{y_{n+1}}{y_n}\right|\le \rho_\theta(\lambda h),
\end{displaymath}
and the stability domain is
$ D_\theta = \left\{ z\in\C,\quad \rho_\theta(z) <1 \right\} $.
  \begin{remark}\label{rem5}
    In \cite{perego-2009} an alternative splitting was introduced:
    \begin{displaymath}
      f=\lambda y = ay + b,\quad 
      a = \dfrac{r}{1+r} \lambda \quad {\rm and}\quad 
      b = \dfrac{1}{1+r} \lambda y,
    \end{displaymath}
    involving the real parameter $r$. We modified that splitting in order to adress more easily the limit where $a=\lambda$ that corresponds to $r\rightarrow \pm\infty $ here whereas it corresponds to $\theta=1$ in our case.
    Nevertheless the results can be translated from one splitting to the other with the change of variable $ r = \theta / (1-\theta)$.
  \end{remark}
The stability function $\rho_\theta(z)$ is defined as in the classical case for linear multistep methods, see \cite[Ch. V.1]{Hairer-ODE-II}.
For a fixed value of $\theta$, the EAB$_k$ scheme applied to the right hand side in equation \eqref{eq:split-f} reads
\begin{displaymath}
  y_{n+1} + c_{\theta,1}(\lambda h) y_n + \ldots + 
  c_{\theta,k}(\lambda h) y_{n-k+1}=0.
\end{displaymath}
The coefficients $c_{\theta,j}$ are explicit (for the EAB$_2$ scheme they are given by
$c_{\theta,1}(z)=-1 - \phi_1(\theta z)z - \phi_2(\theta z)(1-\theta)z$ and $c_{\theta,2}(z)=\phi_2(\theta z)(1-\theta)z$). 
We fix $z\in\C$ and consider the
polynomial $\xi^k + c_{\theta,1}(z)\xi^{k-1} + \ldots +c_{\theta,k}$.
It has $k$ complex roots $\xi^1_\theta(z),\ldots,\xi^k_\theta(z)$.
The stability function is defined as
\begin{displaymath}
  \rho_\theta(z) = \max_{1\le j\le k} |\xi^j_\theta(z)|.
\end{displaymath}
The stability function $\rho_\theta$ is numerically analyzed.
For a given vakue of $\theta$. The software \textit{Maple} is used to compute the 
polynomial roots  on a given grid.
That grid either is a 1D grid of $[-30,0]\subset \R^-$ with size $\Delta x = 0.01$ to study the $A(0)$-stability.
Or it is a 2D cartesian grid in $\C$ to study the stability domain.
In that case the grid size is $\Delta x=0.05$ and the domain is $[-40,2]\times[0,60]$ (roughly 900 000 nodes). The isoline $\rho_\theta(z)=1$ then is constructed with the software \textit{gnuplot}.
\subsection{$A(0)$ stability}
\label{A(0)_stab}
The stability functions $\rho_\theta(z)$ are numerically studied for
$z\in\R^-$. These functions have been plotted for different values of the
parameter $\theta$. The results are depicted on figure \ref{fig:A0-stab}. A
limit $\lim_{-\infty } |\rho_\theta|$ is always observed. The scheme is $A(0)$
stable when this limit is lower than 1. From Figure \ref{fig:A0-stab},
\begin{itemize}
\item EAB$_2$ scheme is $A(0)$ stable if $\theta \ge 0.75$,
\item EAB$_3$ scheme is $A(0)$ stable if $0.88 \le \theta \le 1.9$,
\item EAB$_4$ scheme is $A(0)$ stable if $0.94 \le \theta \le 1.2$.
\end{itemize}
Roughly speaking, $A(0)$ stability holds for the EAB$_k$ scheme if the exact
linear part of $f(t,y)$ in problem \eqref{eq:F1} is approximated with an
accuracy of 75 \%, 85 \% or 95\% for $k=2$, 3 or 4 respectively.
\subsection{$A(\alpha)$ stability}
\begin{figure}[h!]
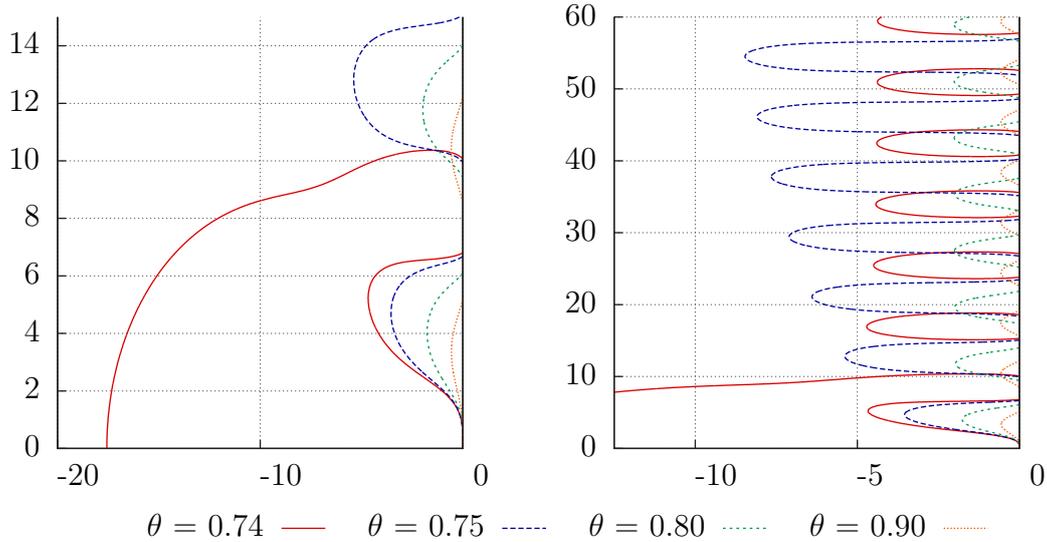

  \centering
  \mbox{
    \input{contour_EAB2-1.tex}
    ~~~~
    \input{contour_EAB2-2.tex}
  }
  \\[10pt]
  \input{contour_EAB2-3.tex}
  \caption{EAB$_2$: isolines $\rho_\theta(z)=1$ for two different ranges. The stability domain $D_\theta$ is
    on the left of the isoline.}
  \label{fig:EAB2-D-stab}
\end{figure}
\begin{figure}[h!]
  \centering
  \mbox{
    \input{contour_EAB3-1.tex}
    ~~~~
    \input{contour_EAB3-2.tex}
  }
  \\[10pt]
  \input{contour_EAB3-3.tex}
  \caption{Same thing as figure \ref{fig:EAB2-D-stab} for the EAB$_3$ scheme.}
  \label{fig:EAB3-D-stab}
\end{figure}
\begin{figure}[h!]
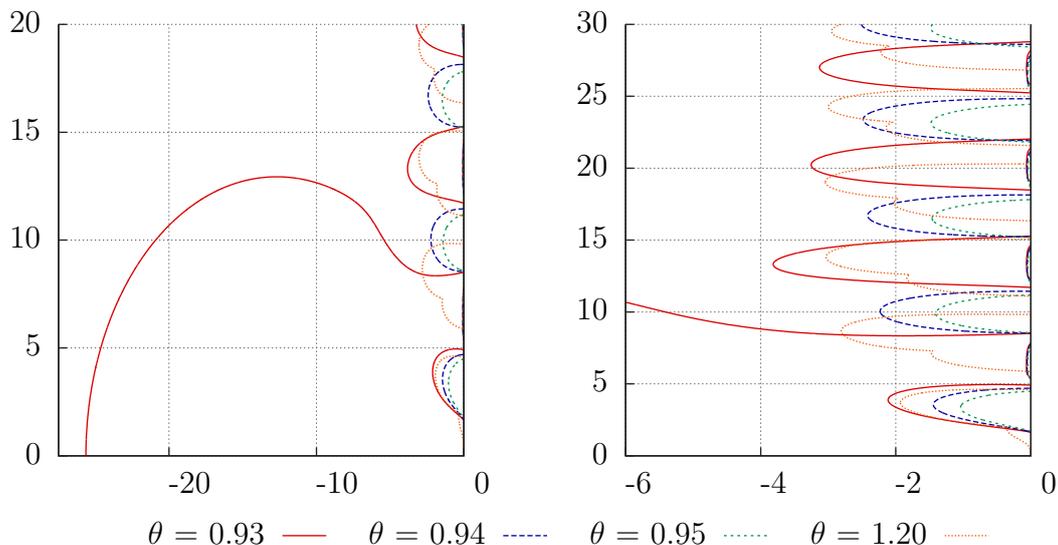

  \centering
  \mbox{
    \input{contour_EAB4-1.tex}
    ~~~~~
    \input{contour_EAB4-2.tex}
  }
  \\[10pt]
  \input{contour_EAB4-3.tex}
  \caption{Same thing as figure \ref{fig:EAB2-D-stab} for the EAB$_4$ scheme.}
  \label{fig:EAB4-D-stab}
\end{figure}
The stability domains $D_\theta$ have been plotted for various values of
$\theta$ taken from figure \ref{fig:A0-stab}. The results are depicted on the
figures \ref{fig:EAB2-D-stab} to \ref{fig:EAB4-D-stab} for $k=2$ to 4
respectively. Each figure shows the isolines $ \rho_\theta(z)=1$. The stability
domain $D_\theta$ is on the left of these curves.
\begin{itemize}
\item Figure \ref{fig:EAB2-D-stab} shows that the EAB$_2$ scheme is $A(\alpha)$
  stable when $\theta=0.75$, 0.8 and 0.9 with $\alpha\simeq 50$, 60 and 80 angle
  degrees respectively.
\item Figure \ref{fig:EAB3-D-stab} displays $A(\alpha)$ stability with
  $\alpha\simeq 60$, 70 and 60 angle degrees for $\theta=0.88$, 0.9 and 1.9
  respectively for the EAB$_3$ scheme.
\item For the EAB$_4$ scheme eventually, $A(\alpha)$ stability holds with an angle
  $\alpha$ approximately of 65, 70 and 60 degrees for $\theta=0.94$, 0.95 and
  1.2 respectively, as shown on figure \ref{fig:EAB4-D-stab}.
\end{itemize}
In all cases, when $A(\alpha)$ stability is observed, the unstable region inside
$\C^-$ is made of a discrete collection of uniformly bounded sets located along
the imaginary axes. Hence, the stability domain $D_\theta$ also contains half
planes of the form $\text{Re} (z) \le a <0$. We conjecture that, when $\theta$
is so that the EAB$_k$ scheme is $A(\alpha)$-stable, then it is also stiff stable.
\subsection{Conclusion}
For explicit linear multistep methods, $A(0)$ stability cannot occur, see
\cite[chapter V.2]{Hairer-ODE-II}. In contrast, EAB$_k$ and I-EAB$_k$ methods exhibit
much better stability properties. When $\theta$ is close enough to 1, they are
$A(\alpha)$ stable and stiff stable. Such stability properties are comparable
with those of implicit linear multistep methods. Basically speaking, these
properties will hold if the stabilization term $a(t,y)$ in \eqref{eq:F2}
approximates the Jacobian of $f(t,y)$ in \eqref{eq:F1} with an absolute discrepancy lower than 25 \%, 10 \% and 5 \% for $k=$2, 3 and 4 respectively.
\\
It is interesting to compare the results on the EAB$_2$ scheme 
with the results in \cite{perego-2009} concerning another explicit exponential scheme of order 2 named AB2*.
The AB2* scheme displays stronger absolute stability properties than the EAB$_2$ scheme.
The AB2* scheme is $A(0)-$stable for $\theta\ge 2/3$ whereas the EAB$_2$ scheme
requires $\theta\ge 0.75$.
The AB2* scheme is $A$-stable for $\theta\ge 1$ whuich property is not achieved by the EAB$_2$ scheme when $\theta\neq 1$.
Extensions of that scheme to higher orders have also been developped in \cite{RLk} that however loses the $A(0)$-stability property.
\section{Positivity properties}
Consider the scalar equation $y' = f(t,y)$,
with $f(t,y)$ chosen so that there exists a function $a(t,y)$ satisfying
\begin{equation}
  \label{eq:cond-posi-f}
  a(t,y)\le 0 \quad \text{and}\quad 
  a(t,y)(y-K_1) \le f(t,y) \le a(t,y)(y-K_2)
\end{equation}
for two constants $K_1\le K_2$.
Then the solution $y(t)$ satisfies $K_1\le y(t) \le K_2$ for $t>0$ provided that $y(0)\in [K_1, K_2]$. A proof for that property is given in \cite{perego-2009}.
It can be important for the time-stepping method to preserve that positivity property.
This is in particular the case for the application in section \ref{sec:num-res} where the gating variables $w(t)$ satisfy the property \eqref{eq:cond-posi-f} 
with $K_1=0$, $K_2=1$ and with $a(t,y)$ defining the stabilizer of the EAB$_k$ scheme.
\\
For the AB$_2^*$ scheme
introduced in \cite{perego-2009}
(that is explicit exponential and with order 2),
that property has been shown to be preserved 
for a constant linear part $a(t,y)=a$ provided that $h \le \log(3)/|a|$ and adding a condition on the initial data.
These results can be extended to the EAB$_2$ and to the  EAB$_3$ schemes.
It provides sufficient condition for the positivity.
However the technique of the proof fails to apply to the EAB$_4$ scheme.
\\
For simplicity in the sequel, we simply denote $\phi_k = \phi_k(ah)$. 
\begin{proposition}
  \label{prop:pos}
  Consider the equation $y'= ay + b(t,y)$ with $a\le 0$ and
  \begin{displaymath}
    -aK_1 \le b(t,y) \le -aK_2,
  \end{displaymath}
  so that the condition (\ref{eq:cond-posi-f})  holds for $f(t,y)=ay + b(t,y)$.
  Denoting $y_n$ the numerical solution for the EAB$_k$ scheme, we have that $y_n \in [K_1, K_2]$ for $n\ge k$ provided that the following conditions are satisfied.
  \begin{itemize}
  \item EAB$_2$ scheme: $ h \le 1/|a|$ and
    \begin{equation}
      \label{eq:cond-pos-eab2}
      C_p K_1 
      \le   \e^{ah}y_1 - h \phi_2 b_0 \le C_pK_2.
    \end{equation}
    with $C_p = e^{ah} + ah \phi_2$.
  \item EAB$_3$ scheme: $h \le \beta_3/|a|$ for a constant $\beta_3 \simeq 0.331$ and
    \begin{equation}
      \label{eq:cond-pos-eab3}
      C_p K_1 
      \le   \e^{ah}y_2 - 2h (\phi_2 + \phi_3) b_1 \le C_pK_2.
    \end{equation}
    with $C_p = e^{ah} + 2 ah \big(\phi_2+\phi_3\big)$.
  \end{itemize}
\end{proposition}
The proof of proposition \ref{prop:pos} will be based on the following result.
\begin{lemma}
  \label{lem:pos}
  Consider the equation $y'=ay + b(t,y)$ with $a\le 0$ 
  as in proposition \ref{prop:pos} and moreover assume that $b(t,y)\ge 0$.
  Its numerical solution with the EAB$_k$ scheme is denoted $y_n$.
  \\
  For the EAB$_2$ scheme,  $y_n \ge 0$  for $n\ge 2$ if:
  \begin{equation}
    \label{eq:pos-cond-EAB2}
    h \le 1/|a| \quad \text{and} \quad 
    0 \le \e^{ah}y_1 - h \phi_2 b_0,
  \end{equation}
  For the EAB$_3$ sceme: $y_n\ge 0$ for $n\ge 3$ if
  \begin{equation}
    \label{eq:pos-cond-EAB3}
    h \le \beta_3/|a| \quad \text{and} \quad 
    0 \le \e^{ah}y_2 - 2 h (\phi_2+\phi_3) b_1,
  \end{equation}
  for a constant $\beta_3 \simeq 0.331$.
\end{lemma}
\begin{proof}[Proof of lemma \ref{lem:pos}]
  In the present case we have $g_n(t,y)= b(t,y)$.
  For the EAB$_2$ scheme we have
  \begin{displaymath}
    y_{n+1} = \underbrace{\big(\e^{ah}y_n - h\phi_2 b_{n-1}\big)}_{:=s_n} ~+~
     \underbrace{h\big( \phi_1 b_n + \phi_2 b_n\big)}_{\ge 0}.
  \end{displaymath}
It suffices to prove that $s_n\ge 0$, which is done by induction.
Assume that $s_n\ge 0$. We have $s_{n+1} = \e^{ah}y_{n+1} -  h\phi_2 b_{n}$ 
and we get
\begin{align*}
  s_{n+1} = \e^{ah} s_n + h b_n
  \big(
  \e^{ah} \phi_1 + \e^{ah} \phi_2 - \phi_2
  \big)
  = \e^{ah} s_n + h b_n \phi_1 (\e^{ah} + ah \phi_2).
\end{align*}
By assumption $\e^{ah} s_n \ge 0$, $b_n\ge 0$ and so $ h b_n \phi_1\ge 0$.
We then must find the condition so that $\e^{ah} + ah \phi_2\ge 0$.
A simple function analysis gives $ ah \in [-1,0] $ 
which is the stability condition in lemma \ref{lem:pos}.
The condition (\ref{eq:pos-cond-EAB2}) states that $s_1\ge 0$.
\\ \\
For the EAB$_3$ scheme we define
\begin{displaymath}
  s_n = \e^{ah}y_n - 2 h b_{n-1} (\phi_2 + \phi_3)
\end{displaymath}
so that we have
\begin{displaymath}
   y_{n+1} = s_n
   ~+~
     h b_n \big( 
     \phi_1 + 3/2 \phi_2  + \phi_3
     \big)
     ~+~
     h b_{n-2} \big( 
      \phi_2  + \phi_3
     \big)
\end{displaymath}
The two terms on the right are non-negative. We will prove by induction that $s_n\ge 0$.
Assume that $s_n\ge 0$, then:
\begin{displaymath}
  s_{n+1} = \e^{ah} s_n 
  + h b_n
  \big(
  \phi_1 \e^{ah}
  +
  \phi_2
  (3/2 \e^{ah} - 2)
  +
  \phi_3
  (\e^{ah} - 2)
  \big)
  + h b_{n-1}
  \big(
  \phi_2/2 + \phi_3
  \big).
\end{displaymath}
In the right hand side, the first and last terms are non-negative.
Thus the condition for $s_{n+1}$ to be non-negative is that
$$\phi_1 \e^{ah}
  +
  \phi_2
  (3/2 \e^{ah} - 2)
  +
  \phi_3
  (\e^{ah} - 2)\ge 0.$$
The study of that inequality gives $h \le \beta_3 / |a|$ and a numerical evaluation of $\beta_3$ is $\beta_3 \simeq 0.331$. With the initial condition $s_2\ge 0$ we recorev (\ref{eq:pos-cond-EAB3}).
\end{proof}
\begin{proof}[Proof of proposition \ref{prop:pos}]
Consider the EAB$_2$ scheme.
We define $u_n = y_n - K_1$.
We have 
\begin{align*}
  u_{n+1} =& (y_n - K_1) + h\Big( 
  \phi_1 \big( a(y_n-K_1) + (b_n + aK_1)\big)  \\ &
  ~~~~~~~~~~~~~~~~~~~~~+~\phi_2\big((b_n + aK_1) - (b_{n-1} + aK_1)\big)
  \Big)\\ =&
  u_n + h\Big( \phi_1 \big(a u_n + \bar{b}_n\big)~+~ 
  \phi_2\big( \bar{b}_n - \bar{b}_{n-1}\big)\Big),
\end{align*}
with $\bar{b}_j = b_j + aK_1 $.
Thus $u_n$ is the numerical solution with the EAB$_2$ scheme of the equation 
$u' = au + \bar b(t,u)$ with $\bar b(t,u)=b(t,u) + aK_1$.
 By assumption $\bar b(t,u)\ge 0$ and the lemma \ref{lem:pos} applies to $u_n$.
It gives the lower bound in condition \eqref{eq:cond-pos-eab2}.
The upper bound is similarly proven by considering $v_n = K_2 - y_n$.
The proof is  the same for the EAB$_3$ scheme.
\end{proof}
\section{Numerical results}
\label{sec:num-res}
We present in this section numerical experiments that investigate the
convergence, accuracy and stability properties of the I-EAB$_k$ and EAB$_k$
schemes. 
The membrane equation in cardiac electrophysiology is considered for two ionic models, the Beeler-Reuter (BR) and to the ten Tusscher
{\it et al.} (TNNP) models. We refer to
\cite{beeler-reuter} and \cite{tnnp} for the definition of the models. The
stiffness of these two models is due to the presence of different time scales
ranging from 1 ms to 1 s, as depicted on figure \ref{fig:TNNP}. The stabilizer
$a_n$ always is a diagonal matrix in this section.
\begin{figure}[h!]
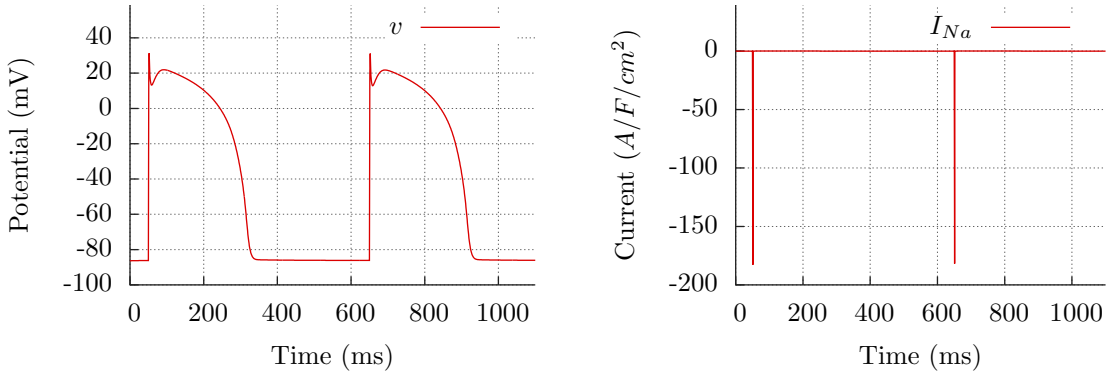

  \centering
  \mbox{
    \input{V_tnnp.tex}
    ~~~~~~
    \input{INa_tnnp.tex}
  }
  \caption{Two consecutive action potentials for the TNNP model: transmembrane
    potential $v$ (left) and the fast sodium current $I_\text{Na}$ (right), that
    is the main component of $\Iion$ during the fast upstroke of the action
    potential.}
  \label{fig:TNNP}
\end{figure}
\\
The membrane equation has the general form, see \cite{Hodgkin52,beeler-reuter,LR2a,tnnp}:
\begin{equation}
  \label{eq:edo-ionic-w}
  \!\!\!\!\!\!
  \diff{w_i}{t}
  = \frac{w_{\infty,i}(v) - w_i}{\tau_i(v)} 
  ,\qquad 
  \diff{c}{t} = q(w,c,v), 
  ,\qquad 
  \diff{v}{t} = - \Iion(w,c,v) + \Ist(t),
\end{equation}
where $w=(w_1,\dots,w_p)\in\R^p$ is the vector of the gating variables,
$c\in \R^q$ is a vector of ionic concentrations or other state variables, and
$v\in\R$ is the cell membrane potential. These equations model the evolution of
the transmembrane potential of a single cardiac cell. The four functions
$w_{\infty,i}(v)$, $\tau_i(v)$, $q(w,c,v)$ and $\Iion(w,c,v)$ are given reaction
terms. They characterize the cell model. The function $\Ist(t)$ is a source
term. It represents a stimulation current.
\\
The formulation \eqref{eq:F2} is recovered with,
\begin{displaymath}
  % \label{eq:vec-Y}
  a(t,y) = 
  \begin{pmatrix}
    -1/\tau(v) & 0& 0 \\
    0 & 0& 0 \\
    0 & 0& 0
  \end{pmatrix}, \quad b(t,y) =
  \begin{pmatrix}
    w_\infty(v)/\tau(v) \\
    q(y) \\
    -\Iion(y) + \Ist(t)
  \end{pmatrix},
\end{displaymath}
for $y = (w,c,v) \in \R^N$ ($N=p+q+1$) and where
$-1/\tau(v) = \diag\left (-1/\tau_i(v)\right)_{i=1\ldots p}$
and
$w_\infty(v)/\tau(v) = \big(
w_{\infty,i}(v)
/\tau_i(v)
\big)_{i=1\ldots p}
$
.
\subsection{Implementation and computational cost}
\label{sec:implementation}
The computation of $y_{n+1}$ with the I-EAB$_k$ and EAB$_k$ schemes requires the data
$y_{n-i}$, $a_{n-i}:=a(t_{n-i}, y_{n-i})$ and $b_{n-i}:=b(t_{n-i}, y_{n-i})$ for
$i=0\ldots k$.
\subsubsection*{EAB$_k$ practical implementation}
Firstly, the $g_{ni} = b_{n-i} + (a_{n-i} -a_n) y_{n-i}$ are updated at each
time step. Then the coefficients $\gamma_{nj}$ in table \ref{tab:gamma-nj} are
computed.
Secondly, the computation of $y_{n+1}$ by formula \eqref{EAB-2} also requires
the computation of the $\phi_j(a_n h)\gamma_{nj}$. This is a matrix-vector
product in general.
\\
In the present case of a diagonal stabilizer, it becomes a scalar-scalar product
per row. The $\phi_j(a_n h)$ are computed on all diagonal entries of $a_n
h$. This computation simply necessitates to compute $\phi_0(a_n h)= \e^{a_n h}$
(one exponential per non zero diagonal entry) thanks to the recursion rule
\eqref{phi_k}.
\\
In general, the relation \eqref{phi_k} can be used to replace the computation of the
$\phi_j(a_n h)\gamma_{nj}$ for $j=0\ldots k$ by the computation of a single
product $\phi_k(a_n h)w_k $. Denoting by $w_1 = a_ny_n + b_n$ and
$w_{j} = \gamma_{nj} + a_n h w_{j-1} $:
\begin{align*}
  \text{EAB}_2~:\quad y_{n+1} &= y_n + h \left[ 
    w_1  
    + \phi_2(a_n h) w_2
  \right],
  \\
  \text{EAB}_3~:\quad y_{n+1} &= y_n + h \left[ 
    w_1 + w_2 /2
    + \phi_3(a_n h)w_3
  \right],   
  \\
  \text{EAB}_4~:\quad y_{n+1} &= y_n + h \left[ 
    w_1 + w_2 / 2 + w_3 / 6
    + \phi_4(a_n h) w_4
  \right]. 
\end{align*}
\subsubsection*{I-EAB$_k$ practical implementation}
In addition, the I-EAB$_k$ method \eqref{sch_IEAB} requires a quadrature rule of
sufficient order to preserve the scheme accuracy and convergence order. We used
the Simpson quadrature rule for the cases $k=2$, 3 and the three point Gaussian
quadrature rule for $k$=4. We point out that $a_n$ is assumed diagonal here so
that the matrix exponentials below actually are scalar exponential.
\\
The I-EAB$_k$ method with Simpson quadrature rule reads,
\begin{displaymath}
  y_{n+1} = \e^{\tilde{g}_1} \left( y_n +  b_n h/6  \right)
  + \left( \tilde{b}_{1} + 4\e^{\delta}\tilde{b}_{1/2}
  \right)h/6,
\end{displaymath}
where (with the notations of section \ref{sec:IEABk}
$\tilde{g}_1 = \tilde{g}_n(t_{n+1})$,
$\delta = \tilde{g}1 - \tilde{g}_n(t_{n}+h/2)$,
$\tilde{b}_{1}= \tilde{b}_n(t_{n+1})$ and
$\tilde{b}_{1/2}= \tilde{b}_n(t_{n}+h/2)$. These coefficients are given for
$k=2$ by
\begin{flalign*}
  \tilde{g}_1 = \left( 3 a_n - a_{n-1} \right) h/2,\qquad &
  \delta = \left( 7 a_n - 3 a_{n-1} \right) h/8,\\
  \tilde{b}_{1} =2b_n-b_{n-1},\qquad & \tilde{b}_{1/2} =(3 b_n - b_{n-1})/2 ,
\end{flalign*}
and for $k=3$ by
\begin{align*}
  \tilde{g}_1  = \left( 23 a_n - 16 a_{n-1} + 5 a_{n-2} \right)h /12,\qquad &
  \delta = \left(29 a_n - 25 a_{n-1}+ 8 a_{n-2} \right)h/24,\\
  \tilde{b}_{1} =3b_n - 3b_{n-1} + b_{n-2},\qquad &
  %% ICI: CHARLES: MODIF
  %% \tilde{b}_{1/2} = \left( 15 b_n - 10 b_{n-1} + 3 b_{n-2} \right)/12,
  \tilde{b}_{1/2} = \left( 15 b_n - 10 b_{n-1} + 3 b_{n-2} \right)/8,
\end{align*}
\\
The I-EAB$_k$ method with the three point Gaussian quadrature rule reads,
\begin{displaymath}
  y_{n+1}  = \e^{\tilde{g}_1} \left( y_n + \frac{h}{18} \left(     
      5 \tilde{b}_l \e^{-\tilde{g}_l} 
      + 8 \tilde{b}_0 \e^{-\tilde{g}_0} 
      + 5 \tilde{b}_r \e^{-\tilde{g}_r} 
    \right)\right), 
\end{displaymath}
with $\tilde{b}_s = \tilde{b}_n(t_s)$, $\tilde{g}_s = \tilde{g}_n(t_s)$ for
$s\in\{l, 0, r\}$ where $t_l = t_n + (1 - \sqrt{3/5}))h/2$, $t_0 = t_n + h/2$,
$t_r = t_n + (1 + \sqrt{3/5}))h/2$ and with
$\tilde{g}_1 = \tilde{g}_n(t_{n+1})$.
These parameters are linear combination of the data $a_{n-i}$, $b_{n-i}$ for
$i=0\ldots k-1$ with fixed coefficients. Formula for $k=4$ follow. The
parameters $\tilde{b}_s$ are given by:
\begin{align*}
  16\,\tilde{b}_0 =& 
  35    \,b_n
  - 35  \,b_{n-1}
  + 21  \,b_{n-2}
  - 5   \,b_{n-3}
  \\
  40\,\tilde{b}_r = &
  ( 95  +179  \sqrt{15} / 15 )  b_n
  - ( 107 +119  \sqrt{15} / 5  )  b_{n-1} \\&\quad 
  + ( 69  +79   \sqrt{15} / 5  )  b_{n-2}
  - ( 17  +59   \sqrt{15} / 15 )  b_{n-3}
\end{align*}
and $\tilde{b}_l$ is the radical conjugate of $\tilde{b}_r$ (the radical
conjugate of $x+\sqrt{y}$ is $x-\sqrt{y}$).
Finally, the parameters $\tilde{g}_s$ definition is:
\begin{align*}
  24/h \;\tilde{g}_1  =& 
  55    \,a_n
  - 59  \,a_{n-1}
  + 37  \,a_{n-2}
  - 9   \,a_{n-3}
  ,
  \\
  384/h \;\tilde{g}_0 = &
  297    \,a_n
  - 187  \,a_{n-1}
  + 107  \,a_{n-2}
  - 25   \,a_{n-3}
  ,
  \\
  200/h \;\tilde{g}_r = & 
  (  797/4  + 45  \sqrt{15}    )  a_n
  - (  2233/12 + 47  \sqrt{15}    )  a_{n-1} 
  \\ & \quad
  + (  1373/12 + 29  \sqrt{15}    )  a_{n-2}
  - (  331/12  + 7   \sqrt{15}    )  a_{n-3} 
  ,
\end{align*}
and $\tilde{g}_l$ is the radical conjugate of $\tilde{g}_r$.
\subsubsection*{Computational cost}
Consider an ODE system \eqref{eq:F1} whose numerical resolution cost is
dominated by the computation of $(t,y)\mapsto f(t,y)$. This might be the case in
general for \textit{``large and complex models''}. For such problems explicit
multistep methods are relevant since they will require one such operation per
time step. In contrast, 
implicit methods, associated to a non linear solver, may necessitate
a lot of these operations, especially for large time steps when convergence is
harder to reach.

In addition, the I-EAB$_k$ and EAB$_k$ schemes need several specific operations.  In
the case of a diagonal function $a(t,y)$ they have been previously described:
the EAB$_k$ require one scalar exponential computation per non zero row of
$a(t,y)$, the I-EAB$_k$ with Simpson rule needs twice more and the I-EAB$_3$ with 3
point Gaussian quadrature rule four times more.
Such a cost is not negligible, but is at worst of same order than computing
$(t,y)\mapsto f(t,y)$ for complex models. For the TNNP model considered here,
computing $(t,y)\mapsto f(t,y)$ costs 50 scalar exponentials whereas the EAB$_k$
implementation adds 7 supplementary scalar exponentials per time step.
In terms of cost per time step, the EAB$_k$ method is rather optimal. The
relationship between accuracy and cost of the EAB$_k$ method has been investigated
in \cite{Cari_2016}: more details are available in section \ref{sec:acc}.
\subsection{Convergence}
\begin{figure}[h!]
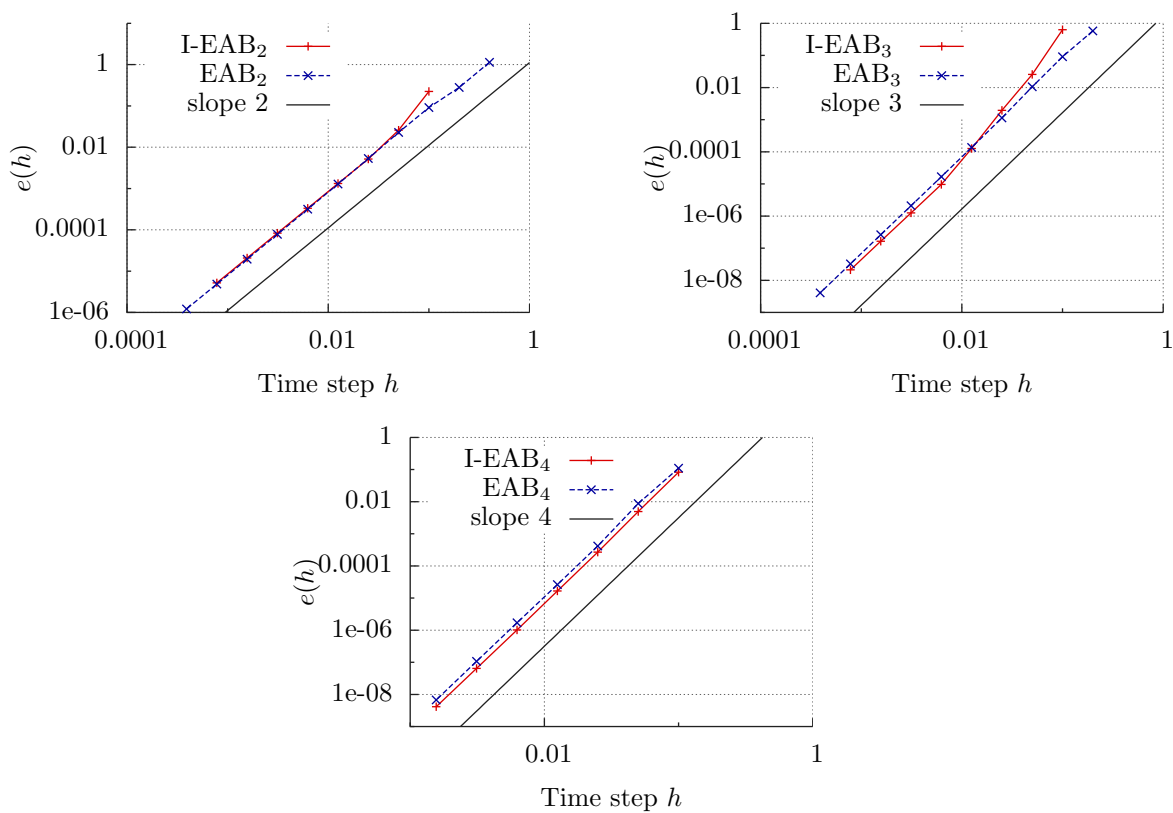

  \centering
  \mbox{
    \input{conv2.tex}
    ~~~~~~~~~~~~~~
    \input{conv3.tex}
  }
  \\[13pt]
  \input{conv4.tex}
  \caption{Relative $L^\infty$ error $e(h)$ 
    for the I-EAB$_k$ and EAB$_k$ schemes, $k=2$, 3 and 4, 
    and for the BR model.}
  \label{fig:The numerical conve}
\end{figure}
For the chosen application, no theoretical solution is available. Convergence
properties are studied by computing a reference solution $y_{ref}$ for a
reference time step $h_{ref}$ with the Runge Kutta 4 scheme. Numerical
solutions $y$ are computed to $y_{ref}$ for coarsest time steps
$h = 2^p h_{ref}$ for increasing $p$.
Any numerical solution $y$ consists in successive values $y_n$ at the time
instants $t_n = n h$. On every interval $(t_{3n}, t_{3n+3})$ the polynomial
$\overline{y}$ of degree at most 3 so that
$\overline{y}(t_{3n+i}) = y_{3n+i}$, $i=0\ldots 4$ is constructed. On $(0,T)$,
$\overline{y}$ is a piecewise continuous polynomial of degree 3. Its values at
the reference time instants $n h_{ref}$ are computed. This provides a
projection $P(y)$ of the numerical solution $y$ on the reference grid. Then
$P(y)$ can be compared with the reference solution $y_{ref}$. The numerical
error is defined by
\begin{equation}
  \label{eq:def-erreur}
  e(h)= \frac{ \max \left|\vref - P(v)\right|}{ \max \left|\vref\right|},
\end{equation}
where the potential $v$ is the last and stiffest component of $y$ in equation
\ref{eq:edo-ionic-w}.
\\
The convergence graphs for the BR model are plotted on figure
\ref{fig:The numerical conve}. All the schemes display the expected asymptotic
behavior $e(h)=O(h^k)$ as $h\to 0$, as proved in theorem \ref{thm:conv}.
\subsection{Stability}
The stiffness of the BR and TNNP models along one cellular electrical cycle
(as depicted on figure \ref{fig:TNNP} has been evaluated in \cite{spiteri}.
The largest negative real part of the eigenvalues of the Jacobian matrix
during this cycle is of $-1170$ and $-82$ for the TNNP and BR models
respectively. This means that the TNNP model is 15 times stiffer than the BR
model ( $ 15 \simeq 1170/82 $).
\\
We want to evaluate the impact of this increase of stiffness in terms of
stability for the EAB$_k$ and I-EAB$_k$ schemes and to provide a comparison with
some other classical time stepping methods. To this aim we consider the
\textit{critical time step} $\Delta t_0$. It is defined as the largest time
step such that the numerical simulation runs without overflow nor non linear
solver failure for $h< \Delta t_0$. The numerical evaluation of $\Delta t_0$
is  easy for explicit methods. For implicit methods, the choice of the
non linear solver certainly impacts $\Delta t_0$. Without considering more
deeply this problem, we just carefully set up the non linear solver, so as to
provide the largest $\Delta t_0$. In practice, we have been using a Jacobian
free Krylov Newton method.
% The critical time steps for implicit methods presented here rather are
% representative than absolute values.
\begin{table}[h!]
  \centering
  \caption{Critical time step $\Delta t_0$}
  \vspace{5pt}
  \mbox{\small{
  \subtable[Classical methods]{%
    \begin{tabular}{lll} \toprule
      & BR & TNNP \\ \midrule
      AB$_2$ & \num{0.124e-1} & \num{0.850e-3} \\
      BDF$_2$ & \num{0.306} & \num{0.158} \\ \addlinespace
      AB$_3$ & \num{0.679e-2} & \num{0.464e-3} \\
      BDF$_3$ & \num{0.362} & \num{0.181} \\ \addlinespace
      AB$_4$ & \num{0.372e-2} & \num{0.255e-3} \\
      RK$_4$ & \num{0.338e-1} & \num{0.255e-2} \\
      BDF$_4$ & \num{0.423} & \num{0.201} \\ \bottomrule
    \end{tabular}%
  }
  \hskip 4em
  \subtable[\mbox{I-EAB$_k${} and  EAB$_k${} exponential methods}]{%
    \begin{tabular}{lll} \toprule
      & BR & TNNP  \\ \midrule
      I-EAB$_2$ &  \num{0.121}& \num{0.103} \\
      EAB$_2$  &  \num{0.424}& \num{0.233} \\ \addlinespace
      I-EAB$_3$ &  \num{0.103}& \num{0.123} \\
      EAB$_3$  &  \num{0.203}& \num{0.108} \\ \addlinespace
      I-EAB$_4$ &  \num{0.133}&\num{0.106}  \\
      EAB$_4$  &  \num{0.122}&\num{0.756e-1} \\ \\ \bottomrule
    \end{tabular}%
  }
}}
  \label{tab_crit_tim2}
\end{table}

Results are on table \ref{tab_crit_tim2}.
The Adams-Bashforth (AB$_k$) and the backward differentiation (BDF$_k$)
methods of order $k$ have been considered, together with the RK$_4$ scheme.
\\
The AB$_k$ and the RK$_4$ schemes have bounded stability domain (see
\cite[p. 243]{Hairer-ODE-II}). Then it is expected for the critical time step
to be divided by a factor close to 15 between the BR and TNNP models. Results
presented in table \ref{tab_crit_tim2} show this behavior.
\\
The BDF$_2$ scheme is $A$-stable whereas the BDF$_3$ and BDF$_4$ are
$A(\alpha)$-stable with large angle $\alpha$ (see
\cite[p. 246]{Hairer-ODE-II}). Hence the critical time step is expected to
remain unchanged between the two models. Table \ref{tab_crit_tim2} shows that
the $\Delta t_0$ actually are divided by approximately 2.
% \begin{table}[htbp!]
%   \caption{Critical time step $\Delta t_0$ for the I-EAB$_k$ and
%   EAB$_k$ schemes.}
%   \label{tab_crit_tim1}
%   \centering
%   \begin{tabular}{lll} \hline
%     & BR & TNNP  \\ \hline
%     I-EAB$_2$ &  \num{0.121}& \num{0.103} \\
%     EAB$_2$  &  \num{0.424}& \num{0.233} \\ \hline
%     I-EAB$_3$ &  \num{0.103}& \num{0.123}  \\
%     EAB$_3$  &  \num{0.203}& \num{0.108} \\ \hline
%     I-EAB$_4$ &  \num{0.133}&\num{0.106}  \\
%     EAB$_4$  &  \num{0.122}&\num{0.756e-1}  \\ \hline
%   \end{tabular}
% \end{table}
\\
The critical time steps for the I-EAB$_k$ and EAB$_k$ models are presented in table
\ref{tab_crit_tim2}. The critical time steps for the I-EAB$_k$ schemes remain
almost unchanged from the BR to the TNNP model. For the EAB$_k$, they are
divided by approximately 2, which behavior is similar as for the BDF$_k$
method.
\\
As a conclusion, for the present application, the EAB$_k$ and I-EAB$_k$ methods are
as robust to stiffness than the implicit BDF$_k$ schemes, though being
explicit. As a matter of fact, section \ref{sec:stab-domain} shows that the
stability domains for the I-EAB$_k$ and EAB$_k$ schemes depend on the discrepancy
between the complete Jacobian matrix and $a(t,y)$. In the present case,
$a(t,y)$ only contains a part of the Jacobian matrix diagonal. It is very
interesting to notice that robustness to stiffness is actually achieved with
this choice. It is finally also interesting to see that the critical time
steps of implicit and exponential methods are of the same order.
\subsection{Accuracy}
\label{sec:acc}
In terms of accuracy, the schemes can be compared using the relative error
$e(h)$ in equation \ref{eq:def-erreur}.
The EAB$_k$ and I-EAB$_k$ schemes can be compared with the AB$_k$ methods only at
very small time steps, because of the lack of stability of AB$_k$ schemes (see
table \ref{tab_crit_tim2}). In table \ref{tab:acc-1} are given the accuracies
of these methods for a given time step $h=10^{-3}$ and for the BR model. Ii is
observed that the same level of accuracy is obtained with AB$_k$ and EAB$_k$ at
fixed $k$. These figures illustrate that inside the asymptotic convergence
region, the EAB$_k$, I-EAB$_k$ and AB$_k$ schemes are equivalent in terms of accuracy.
\begin{table}[h!]
  \caption{Accuracy $e(h)$ for the AB$_k$, I-EAB$_k$ and EAB$_k$ schemes: using the
    BR model and fixed time step $h=10^{-3}$}
  \vspace{5pt}
  \label{tab:acc-1}
  \centering
  \mbox{\small{
  \vspace{5pt}
  \begin{tabular}{cccc} \toprule
    & $k=2$ & $k=3$ & $k=4$ \\ \midrule
    AB$_k$   & \num{5.32e-6} & \num{4.33e-8} & \num{8.69e-10} \\
    I-EAB$_k$ & \num{8.55e-6} & \num{4.44e-8} & \num{7.30e-10} \\
    EAB$_k$  & \num{7.90e-6} & \num{7.00e-8} & \num{1.16e-9}  \\ \bottomrule
  \end{tabular}
}}
\end{table}
% 
% 
% 
% However, such a small time step has been regarded here (for stability
% reasons) that all these numerical solution correspond to high computational
% cost.

\begin{table}[h!]
  \caption{Accuracy for the TNNP model}
  \vspace{5pt}
  \label{acc_tnnp}
  \centering
  \mbox{\small{ \!\!\!\!\!\!\!\!\!\!\!\!\!\!\!\!\!\!\!\!\!\!\!\!\!\!\!\!
  \subtable[EAB$_k$]{%
    \begin{tabular}{cccc}\toprule
      $h$ & $k=2$ & $k=3$ & $k=4$ \\ \midrule
      \num{0.1}   & \num{0.351}   & \num{0.530}   &  \\
      \num{0.05}  & \num{9.01e-2} & \num{5.59e-2} & \num{8.93e-2} \\
      \num{0.025} & \num{2.14e-2} & \num{7.34e-3} & \num{8.34e-3} \\ \bottomrule
    \end{tabular}
  }~~~~~~~~
  %\vskip 1cm
  \subtable[BDF$_k$]{%
    \begin{tabular}{cccc}\toprule
      % $k=2$ & $k=3$ & $k=4$ \\ \midrule
      % &  & \num{0.129} \\ 
      % \num{3.57e-2} & \num{1.15e-2} & \num{1.44e-2} \\ 
      % \num{1.10e-2} & \num{2.58e-3} & \num{2.38e-3} \\ \bottomrule
      $h$ & $k=2$ & $k=3$ & $k=4$ \\ \midrule
      \num{0.1}   &  &  & \num{0.129} \\ 
      \num{0.05}  & \num{3.57e-2} & \num{1.15e-2} & \num{1.44e-2} \\ 
      \num{0.025} & \num{1.10e-2} & \num{2.58e-3} & \num{2.38e-3} \\ \bottomrule
    \end{tabular}
  }}
  }
\end{table}
Comparison at large time steps between the EAB$_k$ and BDF$_k$ for the TNNP
model is shown in table \ref{acc_tnnp}. These figures show that for large time
steps BDF$_k$ is more accurate than EAB$_k$. A gain in accuracy of factor 2.5, 5
and 6 is observed for $h=0.05$ and for $k$=2, 3 and 4 respectively.  However,
compare row 3 for EAB$_k$ ($h=0.025$) with row 2 for BDF$_k$ ($h=0.05$). It
shows that the numerical solutions with an accuracy close to 0.01 are obtained when dividing
the time step by (less than) 2 between BDF$_k$ and EAB$_k$. Meanwhile, EAB$_k$
with $h=0.025$ costs less than BDF$_k$ with $h=0.05$, as developed in section
\ref{sec:implementation}. 
\\
We conclude that EAB$_k$ schemes provide a cheaper way
to compute numerical solutions at large time step for a given accuracy. The
same conclusion also holds for the BR model, see table \ref{acc_br}. A deeper
analysis of the relationship between accuracy and computational cost for the
EAB$_k$ scheme as compared to other methods is available in \cite{Cari_2016}
with the same conclusion.
\begin{table}[h!]
  \caption{Accuracy for the BR model}
  \vspace{5pt}
  \centering
  \mbox{\small{\!\!\!\!\!\!\!\!\!\!\!\!\!\!\!\!\!\!\!\!\!\!\!\!\!\!\!
  \subtable[EAB$_k$]{%
    \begin{tabular}{cccc}\toprule
      $h$ & $k=2$ & $k=3$ & $k=4$ \\ \midrule
      \num{0.2}  & \num{0.284}   & \num{0.516}   &  \\ 
      \num{0.1}  & \num{9.26e-2} & \num{9.17e-2} & \num{0.119} \\ 
      \num{0.05} & \num{8.20e-2} & \num{1.09e-2} & \num{8.96e-3} \\ \bottomrule
    \end{tabular}
  }
  ~~~~~~
  % \vskip 1ex
  \subtable[BDF$_k$]{%
    \begin{tabular}{cccc}\toprule
      $h$ & $k=2$ & $k=3$ & $k=4$ \\ \midrule
      \num{0.2}  & \num{9.74e-2} &  \num{4.09e-2}  &  \num{4.98e-2}  \\
      \num{0.1}  & \num{3.44e-2} &  \num{1.04e-2}  &  \num{1.27e-2}  \\
      \num{0.05} & \num{9.74e-3} &  \num{2.29e-3}  &  \num{2.02e-3}  \\ \bottomrule
    \end{tabular}
  }
}}
  \label{acc_br}
\end{table}

In table \ref{acc_br} are given the accuracies at large time step
% for the EAB$_k$ scheme
now considering the BR model. Comparison with table \ref{acc_tnnp} shows that
accuracy is preserved by dividing $h$ by 2 when switching from the BR to the
TNNP model. As already said, the TNNP model is 15 times stiffer than the BR
model. 
\\
We conclude that the EAB$_k$ schemes also exhibit a large robustness to
stiffness in terms of accuracy. This robustness is equivalent as for the
implicit BDF$_k$ schemes. This is  remarkable for an explicit scheme, as
for the robustness to stiffness in terms of critical time step discussed in
the previous subsection.
  \bibliographystyle{abbrv}
  \bibliography{biblio}

\end{document}